\documentclass[10pt,oneside,leqno]{amsart}
\usepackage{amsxtra}
\usepackage{amsopn}
\usepackage{color}
\usepackage{amsmath,amsthm,amssymb}
\usepackage{amscd}
\usepackage{amsfonts}
\usepackage{latexsym}
\usepackage{verbatim}
\usepackage{lscape}

\definecolor{verde}{rgb}{0.5,.7,.2}

\theoremstyle{plain}
\newtheorem{theorem}{Theorem}[section]
\newtheorem{definition}[theorem]{Definition}
\newtheorem{lemma}[theorem]{Lemma}
\newtheorem{proposition}[theorem]{Proposition}
\newtheorem{corollary}[theorem]{Corollary}
\newtheorem{remark}[theorem]{Remark}
\newtheorem{example}[theorem]{Example}
\newtheorem{question}[theorem]{Question}
\newtheorem{remark-question}[section]{Remark-Question}
\newtheorem{conjecture}[section]{Conjecture}

\begin{document}
\title[Einstein warped $\mathrm{G}_2$ and $\mathrm{Spin}(7)$ manifolds]{Einstein warped $\mathbf{G_2}$ and $\mathbf{Spin(7)}$ manifolds}
\date{\today}
%%%\author{V\'ictor Manero and Luis Ugarte}

\author{V\'ictor Manero}
\address[V. Manero and L. Ugarte]{Departamento de Matem\'aticas\,-\,I.U.M.A.\\
Universidad de Zaragoza\\
Campus Plaza San Francisco\\
50009 Zaragoza, Spain}
\email{vmanero@unizar.es}
\email{ugarte@unizar.es}

\author{Luis Ugarte}

%%%%%%%%%%%%%%%%%%%%%%%%%%%%%%%%%%%%%%%%%%%%%%%%%%%%%%%%%%%%%%%%%%%%%%%%%%%%%%%%%%%%%%
%ABSTRACT%%%%%%%%%%%%%%%%%%%%%%%%%%%%%%%%%%%%%%%%%%%%%%%%%%%%%%%%%%%%%%%%%%%%%%%%%%%%%
%%%%%%%%%%%%%%%%%%%%%%%%%%%%%%%%%%%%%%%%%%%%%%%%%%%%%%%%%%%%%%%%%%%%%%%%%%%%%%%%%%%%%%

\begin{abstract}
In this paper most of the classes of $\mathrm{G}_2$-structures with Einstein induced metric of negative, null or positive scalar curvature
are realized.
This is carried out by means of warped $\mathrm{G}_2$-structures with fiber an Einstein $\mathrm{SU}(3)$ manifold.
The torsion forms of any warped $\mathrm{G}_2$-structure are explicitly described in terms of the torsion forms of the $\mathrm{SU}(3)$-structure and the warping function,
which allows to give characterizations of the principal classes of Einstein warped $\mathrm{G}_2$ manifolds.
Similar results are obtained for Einstein warped $\mathrm{Spin}(7)$ manifolds with fiber a $\mathrm{G}_2$ manifold.
\end{abstract}

%%%%%%%%%%%%%%%%%%%%%%%%%%%%%%%%%%%%%%%%%%%%%%%%%%%%%%%%%%%%%%%%%%%%%%%%%%%%%%%%%%%%%%
%INTRODUCTION%%%%%%%%%%%%%%%%%%%%%%%%%%%%%%%%%%%%%%%%%%%%%%%%%%%%%%%%%%%%%%%%%%%%%%%%%%%%%
%%%%%%%%%%%%%%%%%%%%%%%%%%%%%%%%%%%%%%%%%%%%%%%%%%%%%%%%%%%%%%%%%%%%%%%%%%%%%%%%%%%%%%

\maketitle

\setcounter{tocdepth}{3} \tableofcontents

\begin{section}*{Introduction}

\noindent
The relation between geometric structures (as almost Hermitian or $\mathrm{G}_2$-structures, among others) and Einstein metrics has been deeply studied by many different authors. In particular, one of the most important problems related with this issue is the longstanding conjecture
due to Goldberg~\cite{Go}:

\medskip

\centerline{``\emph{A compact almost K\"ahler Einstein manifold is K\"ahler}''.}

\medskip

Partial affirmative answers have been obtained under some additional curvature conditions. For instance, in~\cite{Se} Sekigawa proved that assuming non-negative scalar curvature the conjecture is true. However, the general case is still open.
Concerning the non-compact version of this conjecture, Apostolov, Draghici and Moroianu found a counterexample which is described in \cite{ADM}. This example consists on a non-compact solvmanifold (solvable Lie group) endowed with a left-invariant almost K\"ahler structure
whose induced metric is Einstein. As the almost complex structure is not integrable, the almost K\"ahler structure is not K\"ahler.

A $\mathrm{G}_2$-structure on a 7-dimensional manifold $M$ consists in a reduction of the structure group of its frame bundle to the Lie group $\mathrm{G}_2$. Equivalently, such structure can be characterized by the existence of a global
non-degenerate 3-form $\varphi$ on $M$. Any $\mathrm{G}_2$-structure has an induced Riemannian metric $g_\varphi$.
When $d\varphi=0$ the manifold $(M,\varphi)$ is called closed $\mathrm{G}_2$ manifold, and if in addition
the 3-form $\varphi$ is  coclosed then it is necessarily parallel with respect to the Levi-Civita connection of $g_\varphi$ \cite{FG}.
Parallel $\mathrm{G}_2$ manifolds are Ricci flat and have holonomy in $\mathrm{G}_2$.
Gibbons, Page and Pope described a $\mathrm{G}_2$-analogue of the Goldberg conjecture in~\cite{GPP}
where they study supersymmetric string solutions on closed $\mathrm{G}_2$-manifolds.
This analogue can be stated as follows:

\smallskip

\centerline{``\emph{A compact Einstein closed $\mathrm{G}_2$ manifold is parallel}''.}

\medskip

In \cite{CI1} Cleyton and Ivanov answer positively to this question.
For the non-compact version, several authors have given partial affirmative answers under some additional conditions.
For example, in \cite{Br} it is shown that every Einstein closed $\mathrm{G}_2$ manifold with non-negative scalar curvature is parallel.
In \cite{CI} the authors proved that Einstein closed $\mathrm{G}_2$-manifolds which are also $\ast$-Einstein are, in fact, parallel.
In \cite{FFM} it is shown that in contrast to the almost K\"ahler case, a seven-dimensional solvmanifold cannot admit any left-invariant closed $\mathrm{G}_2$-structure such that its induced metric is Einstein, unless it is parallel.

%%%\smallskip

Up to this point, a question that naturally arises is the following:
which classes of $\mathrm{G}_2$-structures can induce an Einstein metric?
Our goal in this paper is to show that one can realize most of the classes of $\mathrm{G}_2$-structures
with Einstein induced metric of negative, null or positive scalar curvature (see Table~\ref{tabla-resumen} and Theorem~\ref{resumenG2}). We also study the analogous problem for $\mathrm{Spin}(7)$ manifolds (see Table~\ref{tabla-resumen2} and Theorem~\ref{resumenSpin7}).
For the construction of such structures, we will consider Einstein warped $\mathrm{G}_2$, resp. $\mathrm{Spin}(7)$, manifolds with fiber an Einstein $\mathrm{SU}(3)$,
resp. $\mathrm{G}_2$ manifold. Next we explain in more detail the contents of the paper.

\smallskip

In Section~\ref{sec-SU(3)} we recall some well known results about $\mathrm{SU}(3)$-structures
$(\omega, \psi_+)$ on a 6-dimensional manifold $L$,
as the description of the scalar curvature of the induced metric $g_{\omega, \psi_+}$ and the principal
classes of $\mathrm{SU}(3)$-structures in terms of their torsion forms~\cite{BV,CS}.
Section~\ref{sec-G2} is devoted to general results about $\mathrm{G}_2$-structures $\varphi$ on a 7-dimensional manifold $M$
following~\cite{Br,CS}. We also recall the sixteen Fern\'andez-Gray $\mathrm{G}_2$-classes
$\mathcal{P}$, $\mathcal{X}_i$, $\mathcal{X}_i \oplus \mathcal{X}_j$,
$\mathcal{X}_i \oplus \mathcal{X}_j \oplus \mathcal{X}_k$
and $\mathcal{X}=\mathcal{X}_1 \oplus \mathcal{X}_2 \oplus \mathcal{X}_3\oplus \mathcal{X}_4$,
as well as their description in terms of the torsion forms $\tau_0, \tau_1, \tau_2, \tau_3$
of the $\mathrm{G}_2$-structure.
%%%\cite{FG}
In Section~\ref{subsec-warped-G2}, a class of $\mathrm{G}_2$-structures
on warped products $M=I_f \times L$ with fiber an $\mathrm{SU}(3)$ manifold $L$ is considered,
which provides a natural extension of the well-known usual, exponential and sine
cones (see Proposition~\ref{prop-3-form}).
Different constructions of $G$-structures based on warped products or cones have been studied by many authors
(see for instance \cite{AChFrH,Bar,Besse,BoyerGalicki,CI,FIMU,FR1} and the references therein).
We obtain in Theorem~\ref{torsiones} an explicit description of the torsion forms
of the warped $\mathrm{G}_2$-structure in terms of the torsion forms of the $\mathrm{SU}(3)$-structure
and the warping function $f$.

Our goal in Section~\ref{sec-Einstein-G2} is to construct Einstein 7-manifolds in the different $\mathrm{G}_2$-classes
by means of warped products of certain Einstein $\mathrm{SU}(3)$ manifolds.
In this way explicit Einstein examples with scalar curvatures of different signs are obtained.
In Section~\ref{sec-main-classes} we focus on the principal classes of $\mathrm{G}_2$ manifolds,
giving characterizations for the existence of a parallel, nearly parallel or Einstein locally conformal parallel
warped $\mathrm{G}_2$-structure in terms of the $\mathrm{SU}(3)$ geometry of the fiber.
Such $\mathrm{G}_2$-structures correspond to the classes $\mathcal{P}$, $\mathcal{X}_1$ and $\mathcal{X}_4$, respectively.
For the $\mathrm{G}_2$-class $\mathcal{X}_2\oplus\mathcal{X}_3$ it is proved that
if a warped $\mathrm{G}_2$ manifold $M$ is Einstein then it is parallel
(see Proposition~\ref{2+3-G2}), in particular the
$\mathrm{G}_2$-analogue of the Goldberg conjecture holds for warped $\mathrm{G}_2$ manifolds,
as closed $\mathrm{G}_2$ manifolds constitute the class $\mathcal{X}_2$.

In Section~\ref{sec-coclosed} we obtain Einstein coclosed $\mathrm{G}_2$-structures,
i.e. in the class $\mathcal{X}_1 \oplus \mathcal{X}_3$, on warped products of $\mathrm{SU}(3)$ manifolds of type
$\mathcal{W}_1^+ \oplus \mathcal{W}_1^- \oplus \mathcal{W}_3$, and
apply the construction to the manifold $S^3 \times S^3$ endowed with one of the $\mathrm{SU}(3)$-structures found in~\cite{Sc}.
In Section~\ref{sec-coupled} we construct Einstein $\mathrm{G}_2$ manifolds in different classes starting with a $6$-manifold
endowed with a coupled structure.
Coupled $\mathrm{SU}(3)$-structures were first introduced in \cite{Sal-Milan} and have torsion class $\mathcal{W}_1^-\oplus \mathcal{W}_2^-$,
so they are half-flat and generalize the nearly K\"ahler structures. The twistor space
$\mathcal{Z}$ over a self-dual Einstein $4$-manifold has an Einstein coupled $\mathrm{SU}(3)$-structure~\cite{Tomasiello}, which is used
in~\cite{FR1} to construct a Ricci-flat locally conformal closed $\mathrm{G}_2$ manifold,
i.e. in the class $\mathcal{X}_2 \oplus \mathcal{X}_4$ 
(see~\cite{FR2} for Einstein solvmanifolds in this class with negative scalar curvature).
In Theorems~\ref{coupled-alpha-beta-constantes}, \ref{coupled-alpha-beta-no-constantes}
and~\ref{2+3+4} we construct Einstein $\mathrm{G}_2$ manifolds of negative, null and positive scalar curvature
in the classes $\mathcal{X}_2 \oplus \mathcal{X}_4$, $\mathcal{X}_1 \oplus \mathcal{X}_2\oplus \mathcal{X}_3$,
$\mathcal{X}_1 \oplus \mathcal{X}_3\oplus \mathcal{X}_4$ and $\mathcal{X}_2 \oplus \mathcal{X}_3\oplus \mathcal{X}_4$.
An Einstein $6$-solvmanifold $S$, of negative scalar curvature, is considered in Section~\ref{sec-solvmanifolds} to obtain an Einstein
$\mathrm{G}_2$ manifold on the hyperbolic cosine cone over $S$.

Motivated by the classification problem studied in~\cite{CMS}, in Section~\ref{sec-classification-G2}
we realize most of the $\mathrm{G}_2$-classes
in the Einstein setting with scalar curvature of different signs (see Theorem~\ref{resumenG2}).
More concretely, in the Ricci flat case and in the case of positive scalar curvature,
there exist Einstein warped $\mathrm{G}_2$-structures of every admissible strict type, except possibly for $\mathcal{X}_1\oplus\mathcal{X}_2\oplus\mathcal{X}_4$.
%%%Here we use the term ``strict'' to indicate that the $\mathrm{G}_2$-structure does not belong to any subclass of the given one.
On the other hand, there are Einstein warped $\mathrm{G}_2$-structures with negative scalar curvature of every admissible strict type, except for $\mathcal{X}_2$, $\mathcal{X}_3$, $\mathcal{X}_2\oplus\mathcal{X}_3$, and possibly for $\mathcal{X}_1\oplus\mathcal{X}_2\oplus\mathcal{X}_4$. Table~\ref{tabla-resumen} shows concrete Einstein examples, when they exist, in the different $\mathrm{G}_2$-classes
together with information on the $\mathrm{SU}(3)$ geometry of the fibers.
At the end of Section~\ref{sec-classification-G2}, explicit families of Einstein $\mathrm{G}_2$-structures
with identical Riemannian metric
but having different $\mathrm{G}_2$ type are given (see \cite{AChFrH,Br,Grigorian,spiro,Lin} for related results).

Section~\ref{sec-Spin(7)} is devoted to warped $\mathrm{Spin}(7)$ manifolds $(N=I_f \times M,\phi)$
with fiber a $\mathrm{G}_2$ manifold $(M,\varphi)$.
In Theorem~\ref{torsionesSpin7} we describe the torsion forms $\lambda_1, \lambda_5$ of the $\mathrm{Spin}(7)$-structure $\phi$
in terms of the torsion forms of the fiber, which allows to give characterizations for the existence of a parallel
or an Einstein locally conformal parallel
warped $\mathrm{Spin}(7)$-structure in terms of the $\mathrm{G}_2$ geometry of the fiber.
In Section~\ref{sec-Einstein-Spin(7)} Einstein 8-manifolds in the different $\mathrm{Spin}(7)$-classes, i.e.
$\mathcal{P}$, $\mathcal{Y}_1$, $\mathcal{Y}_2$ and the general class $\mathcal{Y}=\mathcal{Y}_1 \oplus \mathcal{Y}_2$, 
are constructed.
For zero or positive scalar curvatures, there are Einstein warped $\mathrm{Spin}(7)$-structures of every admissible strict type, whereas
for negative scalar curvature there are Einstein warped $\mathrm{Spin}(7)$-structures of every admissible strict type, except for $\mathcal{Y}_2$
(see Theorem~\ref{resumenSpin7} and Table~\ref{tabla-resumen2}).
\end{section}

%%%%%%%%%%%%%%%%%%%%%%%%%%%%%%%%%%%%%%%%%%%%%%%%%%%%%%%%%%%%%%%%%%%%%%%%%%%%%%%%%%%%%%,
% SU(3)-STRUCTURES%%%%%%%%%%%%%%%%%%%%%%%%%%%%%%%%%%%%%%%%%%%%%%%%%%%%%
%%%%%%%%%%%%%%%%%%%%%%%%%%%%%%%%%%%%%%%%%%%%%%%%%%%%%%%%%%%%%%%%%%%%%%%%%%%%%%%%%%%%%%

\begin{section}{$\mathrm{SU}(3)$-structures}\label{sec-SU(3)}

\noindent
An $\mathrm{SU}(3)$-structure on a 6-dimensional manifold $L$ consists of a triple $(g, J, \Psi)$
such that $g$ is a Riemannian metric, $J$ is an almost complex structure compatible with the metric,
and $\Psi$ is a complex volume form satisfying
\begin{equation*}
\frac{3}{4}i\, \Psi \wedge \overline{\Psi} = \omega^3,
\end{equation*}
where $\omega$ is the
%%%K\"ahler
fundamental form associated to the almost Hermitian structure $(g, J)$.
Note that an $\mathrm{SU}(3)$-structure on a 6-dimensional manifold $L$ can be described by
the pair $(\omega, \psi_+)$, where $\psi_+$ is
the real part of the complex volume form $\Psi$.
Indeed, for the imaginary part $\psi_-$ of the form $\Psi$ one has that $\psi_-=J \psi_+$,
so $\psi_-$ is determined by $\psi_+$ and the almost complex structure $J$
(see \cite{Hitchin}).
%%%the triple $(\omega, \psi_+, \psi_-)$ with $\psi_+$ and $\psi_-$
%%%the real and the imaginary part of the complex volume form $\Psi$, respectively.
We will denote by $g_{\omega, \psi_+}$ the Riemannian metric induced by the $\mathrm{SU}(3)$-structure.

As it is described in \cite{BV}, the intrinsic torsion of an $\mathrm{SU}(3)$-structure can be given in terms of the derivatives of the forms $\omega$, $\psi_+$ and $\psi_-$. Consider the natural action of the group $\mathrm{SU}(3)$ on the spaces $\Omega^p(L)$ of differential $p$-forms on $L$,
and more concretely, the $\mathrm{SU}(3)$ irreducible subspaces of $\Omega^2(L)$ and $\Omega^3(L)$.
One has the following decompositions \cite{BV,CS}:
\begin{equation*}
\Omega^2(L) = \Omega^2_1(L) \oplus \Omega^2_6(L) \oplus \Omega^2_{8}(L),
\end{equation*}
where
\begin{equation*}
\begin{aligned}
\Omega^2_1(L) & = \{f\,\omega \mid f \in \mathcal{C}^{\infty}(L) \}, \\[4pt]
\Omega^2_6(L) & = \{ \ast_6 J (\alpha \wedge \psi_+) \mid \alpha \in \Omega^1(L) \} = \{ \beta \in \Omega^2(L) \mid J\beta=-\beta \}, \\[4pt]
\Omega^2_{8}(L) & = \{ \beta \in \Omega^2(L) \mid \beta \wedge \psi_+=0,\  \ast_6 J \beta = -\beta \wedge \omega \}
= \{ \beta \in \Omega^2(L) \mid J \beta = \beta,\ \beta \wedge \omega^2=0 \},
\end{aligned}
\end{equation*}
and
\begin{equation*}
\Omega^3(L) = \Omega^3_{1_+}(L) \oplus \Omega^3_{1_-}(L) \oplus \Omega^3_{6}(L) \oplus \Omega^3_{12}(L)
\end{equation*}
with
\begin{equation*}
\begin{aligned}
\Omega^3_{1_\pm}(L) & = \{f\, \psi_{\pm}  \mid f \in \mathcal{C}^{\infty}(L) \}, \\[4pt]
\Omega^3_6(L) & = \{ \alpha \wedge \omega \mid \alpha \in \Omega^1(L) \} = \{ \gamma \in \Omega^3(L) \mid  \ast_6 J\gamma=\gamma \},\\[4pt]
\Omega^3_{12}(L) & = \{ \gamma \in \Omega^3(L) \mid \gamma \wedge \omega =0, \ \gamma \wedge \psi_{\pm}=0 \}.
\end{aligned}
\end{equation*}
Here, $\ast_6$ denotes the Hodge star operator, and
$\Omega^p_k(L)$ is the $\mathrm{SU}(3)$ irreducible space of $p$-forms of dimension $k$ at every point.
The decomposition on the other degrees is obtained via the isomorphism described by the Hodge star operator $\ast_6$,
i.e. $\ast_6\, \Omega^p_k(L) \cong \Omega^{6-p}_k(L)$.

Thus, the differentials of $\omega, \psi_+$ and $\psi_-$ can be decomposed into summands belonging
to the $\mathrm{SU}(3)$ invariant spaces as follows:
\begin{equation}\label{SU3torsion}
\begin{aligned}
& d\omega  = -\frac{3}{2}\sigma_0\, \psi_++\frac{3}{2}\pi_0\, \psi_-+\nu_1 \wedge \omega + \nu_3, \\[4pt]
& d\psi_+  = \pi_0\, \omega^2+\pi_1 \wedge \psi_+ - \pi_2 \wedge \omega, \\[4pt]
& d\psi_-  = \sigma_0\, \omega^2+ \pi_1 \wedge \psi_- - \sigma_2 \wedge \omega, \\
\end{aligned}
\end{equation}
where $\sigma_0, \pi_0 \in \mathcal{C}^\infty(L)$, $\pi_1, \nu_1 \in \Omega^1(L)$, $\pi_2, \sigma_2 \in \Omega^2_8(L)$
and $\nu_3 \in \Omega^3_{12}(L)$ are called the \emph{torsion forms}.
Note that in the last equality, $\pi_1 \wedge \psi_- = J\pi_1 \wedge \psi_+$ accordingly to \cite{BV}.

Bedulli and Vezzoni derived the Ricci tensor of the metric $g_{\omega, \psi_+}$ induced by the $\mathrm{SU}(3)$-structure in terms of the torsion forms.
In \cite[Theorem 3.4]{BV}, they find the following expression for the scalar curvature:
\begin{equation}\label{scal6}
Scal(g_{\omega, \psi_+}) =  \frac{15}{2} \pi_0^2+ \frac{15}{2} \sigma_0^2+2d^{\ast_6}\pi_1+2d^{\ast_6}\nu_1-|\nu_1|^2
- \frac{1}{2}|\sigma_2|^2- \frac{1}{2}|\pi_2|^2- \frac{1}{2} |\nu_3|^2+4\langle \pi_1, \nu_1 \rangle.
\end{equation}
Here, $d^{\ast_6}$ denotes the codifferential, i.e. the adjoint of the exterior derivative with respect to the metric.

As it is described in \cite{CS} the torsion of an $\mathrm{SU}(3)$-structure, namely $T$, lies in the space
\begin{equation*}
T \in \mathcal{W}_1^{\pm} \oplus \mathcal{W}_2^{\pm} \oplus \mathcal{W}_3 \oplus \mathcal{W}_4 \oplus \mathcal{W}_5,
\end{equation*}
where $\mathcal{W}_i$ are the irreducible components under the action of the group $\mathrm{SU}(3)$.
The spaces $\mathcal{W}_i$ are related to the torsion forms by Table~\ref{classes}.

\begin{table}[h]\label{classes}
\caption{\textbf{Principal classes of $\mathrm{SU}(3)$-structures}}
\begin{center}
  \begin{tabular}{ |c | c |}
    \hline
    Class & Non-zero torsion form \\ \hline \hline
     $\{0\}$ & --  \\ \hline
    $\mathcal{W}_1^+ $ & $\pi_0$  \\ \hline
    $\mathcal{W}_1^- $ & $\sigma_0$  \\ \hline
    $\mathcal{W}_2^+ $ & $\pi_2$  \\ \hline
    $\mathcal{W}_2^- $ & $\sigma_2$  \\ \hline
    $\mathcal{W}_3 $ & $  \nu_3 $  \\  \hline
    $\mathcal{W}_4 $ & $\nu_1$ \\ \hline
    $\mathcal{W}_5 $ & $\pi_1$ \\ \hline
  \end{tabular}
\end{center}
\end{table}

Hence, torsion forms provide a useful tool to describe the principal classes of $\mathrm{SU}(3)$-structures. For instance, $\mathrm{SU}(3)$-structures with zero torsion are called integrable, or Calabi-Yau, their holonomy is contained in $\mathrm{SU}(3)$ and they are Ricci flat.
The $\mathrm{SU}(3)$-structures in the class $\mathcal{W}_1^-$ are nearly K\"ahler. They are Einstein and
all the torsion forms vanish except for $\sigma_0$.
There are only finitely many homogeneous nearly K\"ahler manifolds \cite{Butruille} and new complete inhomogeneous
examples on $S^6$ and $S^3 \times S^3$ are found recently in~\cite{FosHan}.
Other well known $\mathrm{SU}(3)$-structures are the half-flat structures, for which
$\pi_0=\pi_1=\nu_1=\pi_2=0$, and the nearly half-flat structures, characterized by $\pi_1=\nu_1=\sigma_2=0$.
Half-flat structures were first considered in~\cite{Hitchin-Bilbao} (see also~\cite{CS}) and
the class of nearly half-flat structures was introduced in~\cite{FIMU}, and these structures
can be evolved to a parallel and to a nearly parallel $\mathrm{G}_2$-structure, respectively.

In this paper the $\mathrm{SU}(3)$-structures in the classes $\mathcal{W}_1^+ \oplus \mathcal{W}_1^- \oplus \mathcal{W}_3$
and $\mathcal{W}_1^-\oplus \mathcal{W}_2^-$ will play a role
in the construction of Einstein $\mathrm{G}_2$ manifolds (see Sections~\ref{sec-coclosed} and~\ref{sec-coupled}).
The structures in the first class are characterized by $\pi_1=\nu_1=\pi_2=\sigma_2=0$, and the structures
in the second class are known as coupled $\mathrm{SU}(3)$-structures.
Coupled $\mathrm{SU}(3)$-structures were first introduced in \cite{Sal-Milan} (see also \cite{FR1})
and they are characterized by the condition $d \omega = -\frac{3}{2}\sigma_0\, \psi_+$,
which is equivalent to the vanishing of all the torsion forms except $\sigma_0$ and $\sigma_2$.
Thus, coupled structures are half-flat and they generalize the nearly K\"ahler structures.

\smallskip

We end this section recalling some well-known identities concerning $\mathrm{SU}(3)$-structures that will be useful in the next sections.

\begin{lemma}\label{lemmaprop}
Consider an $\mathrm{SU}(3)$-structure $(\omega, \psi_+, \psi_-)$ on a $6$-manifold $L$.
Then, for any $1$-form $\tau \in \Omega^1(L)$ the following identities hold:
\begin{itemize}
\item $\ast_6(\tau \wedge \omega) \wedge \omega= \ast_6(\tau \wedge \psi_+)\wedge \psi_+=
\ast_6(\tau \wedge \psi_-)\wedge \psi_-= 2 \ast_6 \tau$,
\item $\ast_6(\tau \wedge \psi_+)\wedge \psi_-= -\ast_6(\tau \wedge \psi_-)\wedge \psi_+= - \tau \wedge \omega^2$.
\end{itemize}
\end{lemma}

\begin{proof}
Let $\{e^1, \dots, e^6\}$ be a basis adapted to the $\mathrm{SU}(3)$-structure, i.e. a local
orthonormal basis
such that the forms $\omega, \psi_+$ and $\psi_-$ have the following expressions
$$
\omega  = e^{12}+e^{34}+e^{56}, \quad\ \
\psi_+  = e^{135}-e^{146}-e^{236}-e^{245}, \quad\ \
\psi_-  = e^{136}+e^{145}+e^{235}-e^{246}.
$$
Here we denote by $e^{ij}$, resp. $e^{ijk}$, the wedge product $e^i \wedge e^j$, resp. $e^i \wedge e^j \wedge e^k$.
Now, a generic 1-form on $L$ can be written locally as $\tau=\sum_{i=1}^7a_i e^i$,
with $a_i \in \mathcal{C}^\infty(L)$, and the result follows by a direct calculation.
\end{proof}

\end{section}

%%%%%%%%%%%%%%%%%%%%%%%%%%%%%%%%%%%%%%%%%%%%%%%%%%%%%%%%%%%%%%%%%%%%%%%%%%%%%%%%%%%%%%
%2. G2-STRUCTURES%%%%%%%%%%%%%%%%%%%%%%%%%%%%%%%%%%%%%%%%%%%%%%%%%%%%%
%%%%%%%%%%%%%%%%%%%%%%%%%%%%%%%%%%%%%%%%%%%%%%%%%%%%%%%%%%%%%%%%%%%%%%%%%%%%%%%%%%%%%%

\begin{section}{$\mathrm{G}_2$-structures}\label{sec-G2}

\noindent
A $\mathrm{G}_2$-structure on a 7-dimensional manifold $M$ consists in a reduction of the structure group of its frame bundle to the Lie group $\mathrm{G}_2$. Equivalently, the existence of such structure can be characterized by the existence of a global
non-degenerate 3-form $\varphi$ on $M$ which can be locally written as
\begin{equation}\label{G2can}
\varphi= \, e^{127}+e^{347}+e^{567}+e^{135}-e^{146}-e^{236}-e^{245},
\end{equation}
where $\{e^1,\dots, e^7\}$ is a local basis of 1-forms on $M$.
The presence of a $\mathrm{G}_2$-structure on a manifold defines a Riemannian metric $g_{\varphi}$ which can be obtained as
\begin{equation*}
g_{\varphi}(e_i, e_j) = \frac{1}{6} \sum_{k\,l} \varphi(e_i, e_k, e_l)\, \varphi(e_j, e_k, e_l).
\end{equation*}

Let $(M,\varphi)$ be a $\mathrm{G}_2$ manifold. Then, the group $\mathrm{G}_2$ acts on the space $\Omega^p(M)$ of differential $p$-forms
on the manifold $M$. This action is irreducible on $\Omega^1(M)$ and $\Omega^6(M)$, but it is reducible for $\Omega^p(M)$ with $2 \leq p \leq 5$.
Since the Hodge star operator $\ast_7$ induces an isomorphism between the spaces of $p$-forms and $(7-p)$-forms on $M$,
we only need to describe the decompositions for $p=2$ and $3$. In~\cite{Br} it is shown that
the $\mathrm{G}_2$ irreducible decompositions for $p=2$ and $3$ are
\begin{equation*}
\Omega^2(M) = \Omega^2_7(M) \oplus \Omega^2_{14}(M),
\end{equation*}
where
\begin{equation*}
\begin{aligned}
\Omega^2_7(M) & = \{\ast_7(\alpha \wedge \ast_7\varphi) \mid \alpha \in \Omega^1(M) \}, \\
\Omega^2_{14}(M) & = \{ \beta \in \Omega^2(M) \mid \beta \wedge \varphi = - \ast_7\beta \}
= \{ \beta \in \Omega^2(M) \mid \beta \wedge \ast_7\varphi = 0 \},
\end{aligned}
\end{equation*}
and
\begin{equation*}
\Omega^3(M) = \Omega^3_1(M) \oplus \Omega^3_{7}(M)  \oplus \Omega^3_{27}(M),
\end{equation*}
with
\begin{equation*}
\begin{aligned}
\Omega^3_1(M) & = \{f \varphi \mid f \in \mathcal{C}^\infty(M)  \}, \\
\Omega^3_7(M) & = \{\ast_7(\alpha \wedge \varphi) \mid \alpha \in \Omega^1(M) \}, \\
\Omega^3_{27}(M) & = \{ \gamma \in \Omega^3(M) \mid \gamma \wedge \varphi =0, \ \gamma \wedge \ast_7\varphi =0  \},
\end{aligned}
\end{equation*}
where $\Omega^p_{k}(M)$ denotes a $\mathrm{G}_2$ irreducible space of $p$-forms of dimension $k$ at every point.
Note that the description on the other degrees are obtained via the isomorphism described by the Hodge star operator,
i.e. $\ast_7\,\Omega^p_{k}(M) \cong \Omega^{7-p}_{k}(M)$.

As it is pointed out in \cite{Br},
it is useful to recognize the scaling factors that the isomorphisms between these $\mathrm{G}_2$ irreducible spaces introduce. For example, for any $\kappa \in \Omega^1(M)$ one has
\begin{equation}\label{G2prop}
\begin{aligned}
\ast_7\big(\ast_7(\kappa \wedge \varphi)\wedge \varphi\big)= & -4 \kappa, \\[2pt]
\ast_7\big( \ast_7 ( \kappa \wedge \ast_7 \varphi) \wedge \ast_7 \varphi \big) =& \ 3\kappa.
\end{aligned}
\end{equation}

%%%%%%%%%%%%%%%%%%%%%%%%%%%%%%%%%%%%%%%%%%%%%%%%%%%%%%
%%%%%%%%%%%%%%%%%%%%%%%%%%%%%%%%%%%%%%%%%%%%%%%%%%%%%%

\smallskip

The $\mathrm{G}_2$ type decomposition of forms on $M$ allows to express the exterior derivative of $\varphi$ and $\ast_{7} \varphi$ as follows
\begin{equation}\label{torsionforms}
\begin{aligned}
d\varphi = & \,\, \tau_0 \ast_7\!\varphi + 3\, \tau_1 \wedge \varphi + \ast_7\, \tau_3,\\[2pt]
d \ast_7\!\varphi = & \,\, 4\, \tau_1 \wedge \ast_7 \varphi + \tau_2 \wedge \varphi ,
\end{aligned}
\end{equation}
where $\tau_0 \in \mathcal{C}^\infty(M)$, $\tau_1 \in \Omega^1(M)$, $\tau_2 \in \Omega^2_{14}(M)$ and $\tau_3 \in \Omega^3_{27}(M)$ are called the \emph{torsion forms} of the $\mathrm{G}_2$-structure.

According to \cite{FG} the covariant derivative of $\varphi$ can be decomposed into four irreducible components, namely $X_1, X_2, X_3$ and $X_4$.
Thus, a $\mathrm{G}_2$-structure is said to be of type $\mathcal{P}, \mathcal{X}_i, \mathcal{X}_i \oplus \mathcal{X}_j, \mathcal{X}_i \oplus \mathcal{X}_j \oplus \mathcal{X}_k $ or $\mathcal{X}$ if the covariant derivative $\nabla^{g_\varphi} \varphi$ lies in $\{0\}, X_i, X_i \oplus X_j, X_i \oplus X_j \oplus X_k $ or $X=X_1 \oplus X_2 \oplus X_3 \oplus X_4$, respectively.
Hence, there exist 16 different classes of $\mathrm{G}_2$-structures.
These classes can be described in terms of the behavior of the torsion forms $\tau_0, \tau_1, \tau_2, \tau_3$~\cite{CS}.
In Table~\ref{classes-G2}
the principal Fern\'andez-Gray classes of $\mathrm{G}_2$-structures are given.

\bigskip

\begin{table}[h]\label{classes-G2}
\caption{\textbf{Principal classes of $\mathrm{G}_2$-structures}}
\begin{center}
  \begin{tabular}{ |c | c | l |}
    \hline
    Class & Torsion forms & Structure\\ \hline \hline
    $\mathcal{P}$ & $\tau_0= \tau_1= \tau_2 = \tau_3 = 0$ & Parallel \\ \hline
    $\mathcal{X}_1$ & $\tau_1= \tau_2 = \tau_3 = 0$ & Nearly parallel \\ \hline
    $\mathcal{X}_2$ & $\tau_0= \tau_1 = \tau_3 = 0$ & Closed \\ \hline
    $\mathcal{X}_3$ & $\tau_0= \tau_1 = \tau_2 = 0$ & Coclosed of pure type \\ \hline
    $\mathcal{X}_4$ & $\tau_0= \tau_2 = \tau_3 = 0$ & Locally conformal parallel \\ \hline
    $\mathcal{X}_1 \oplus \mathcal{X}_3$ & $ \tau_1 = \tau_2 = 0 $ & Coclosed  \\
    \hline
  \end{tabular}
\end{center}
\end{table}

Hence, torsion forms constitute a useful tool to describe different $\mathrm{G}_2$-structures.
Moreover, as it was shown by Bryant in \cite{Br}, one can also describe the scalar curvature
of a $\mathrm{G}_2$ manifold in terms of its torsion forms~by
\begin{equation}\label{scal7}
Scal(g_{\varphi})=12 \, d^{\ast_7} \tau_1 + \frac{21}{8}\, \tau^2_0 + 30\, |\tau_1|^2 -\frac{1}{2}\, |\tau_2|^2-\frac{1}{2}\, |\tau_3|^2,
\end{equation}
where $d^{\ast_7}$ is the codifferential with respect to the metric $g_{\varphi}$ on $M$.

The geometry of $\mathrm{G}_2$-structures in the different classes above has been studied by many authors.
Parallel $\mathrm{G}_2$ manifolds have holonomy in $\mathrm{G}_2$ and they are Ricci-flat. Examples of
manifolds with $\mathrm{G}_2$ holonomy are constructed in \cite{Br0,BS,Joyce1}.
On the other hand, any (strict) nearly parallel $\mathrm{G}_2$ manifold is Einstein with positive scalar
curvature \cite{Friedrich-K-M-S}. The classification of $\mathrm{G}_2$ manifolds, initiated in~\cite{FG},
was completed in~\cite{CMS} both in the non-compact and compact cases.
In Section~\ref{sec-classification-G2}  we realize most of the $\mathrm{G}_2$-classes
in the Einstein setting with scalar curvature of different signs.

\end{section}

%%%%%%%%%%%%%%%%%%%%%%%%%%%%%%%%%%%%%%%%%%%%%%%%%%%%%%
%%%%%%%%%%%%%%%%%%%%%%%%%%%%%%%%%%%%%%%%%%%%%%%%%%%%%%

\begin{section}{Warped $\mathrm{G}_2$-structures}\label{subsec-warped-G2}

\noindent In this section we consider a class of $\mathrm{G}_2$-structures on warped products with fiber an $\mathrm{SU}(3)$ manifold, and
we obtain an explicit description of the torsion forms
of the warped $\mathrm{G}_2$-structure in terms of the torsion forms of the $\mathrm{SU}(3)$-structure.

The presence of an $\mathrm{SU}(3)$-structure on a 6-dimensional manifold provides a way to obtain 7-dimensional manifolds endowed with $\mathrm{G}_2$-structures. Indeed, consider $L$ a 6-dimensional manifold endowed
with an $\mathrm{SU}(3)$-structure $(\omega, \psi_+, \psi_-)$.
Let $M$ be the Riemannian product $M=\mathbb{R} \times L$, and denote by
$$
p\colon M \longrightarrow \mathbb{R},\quad\quad q\colon M \longrightarrow L,
$$
the projections. Then, the 3-form
$$
\varphi =  q^*(\omega) \wedge p^*(dt) + q^*(\psi_+),
$$
where $t$ is the coordinate on $\mathbb{R}$, defines a $\mathrm{G}_2$-structure on $M$.
In the following, we will identify $\omega, \psi_+$ and $\psi_-$ with their pullbacks onto $M$.

We will consider a slightly more general class of $\mathrm{G}_2$-structures
given by the warped product construction.
Let $(B, g_B)$ and $(F, g_F)$ be two Riemannian manifolds, and let $f$ be a nowhere vanishing
differentiable function on $B$.
In this paper we suppose that $f$ is never a constant function.
Denote by $p$ and $q$ the projections of $B \times F$ onto $B$ and $F$,
respectively. Recall that the warped product, namely $M = B \times_f F$, is the product manifold $B \times F$
endowed with the metric $g$ given by
\begin{equation*}
g = f^2\, q^*(g_F) +  p^*(g_B).
\end{equation*}
The manifold $B$ is called the base of $M$, $F$ the fiber, and the warped product is called trivial
if $f$ is a constant function.

In what follows, we consider $F=L$ and a 1-dimensional base $B$. More concretely, $B=I_f \subset \mathbb{R}$ is an open interval
where the function $f(t)$ does not vanish. In the next result we introduce the class of $\mathrm{G}_2$-structures
that will be studied.

\begin{proposition}\label{prop-3-form}
Let $(L, \omega, \psi_+, \psi_-)$ be an $\mathrm{SU}(3)$ manifold and consider functions
$f, \alpha, \beta \colon I_f \longrightarrow \mathbb{R}$, with $\alpha^2(t) + \beta^2(t) =1$.
Then, the form on $M = I_f \times L$ given by
\begin{equation}\label{3-form}
\varphi=   f^2(t)\, \omega \wedge dt + f^3(t)\big( \alpha(t) \psi_+ - \beta(t) \psi_-  \big)\\
\end{equation}
defines a family of $\mathrm{G}_2$-structures whose induced metric is
\begin{equation*}
g_{\varphi}=f^2(t)\,g_{\omega,\psi_+}+dt^2.
\end{equation*}
\end{proposition}

\begin{proof}
Consider $\{e^1,\dots, e^6\}$ a local orthonormal basis of 1-forms for which the $\mathrm{SU}(3)$-structure has its canonical expression.
Then, with respect to the basis
\begin{equation*}
\{h^1,\dots,h^7\}=\{f(t)e^1,\dots,f(t)e^4,f(t)\big(\alpha(t)e^5-\beta(t)e^6\big),f(t)\big(\beta(t)e^5+\alpha(t)e^6\big), dt\}
\end{equation*}
the 3-form $\varphi$ can be written as in \eqref{G2can}, and therefore $\{h^1,\dots, h^7\}$ is a local orthonormal basis for
the metric~$g_{\varphi}$. Thus,
\begin{equation*}
g_{\varphi}=\sum_{i=1}^7 h^i \otimes h^i=f^2(t) \sum_{i=1}^6 e^i \otimes e^i+ dt \otimes  dt = f^2(t)\, g_{\omega,\psi_+} + dt^2.
\end{equation*}
\end{proof}

According to the previous proposition, if $(L,\omega, \psi_+, \psi_-)$ is an $\mathrm{SU}(3)$ manifold,
then the $\mathrm{G}_2$ manifold $M=I_f \times L$ with $\varphi$ described in \eqref{3-form} is precisely the warped product manifold
$M = I_f \times_{f} L$.
In what follows, any such $\mathrm{G}_2$-structure $\varphi$ 
will be called \emph{warped $\mathrm{G}_2$-structure},
and we will refer to the pair $(M = I_f \times L, \varphi)$ as a \emph{warped $\mathrm{G}_2$ manifold}.
Notice that the warped $\mathrm{G}_2$-structure generalizes the well-known ideas of cone and sine-cone that appear in the
literature.

\smallskip

Next we will obtain an explicit description of the torsion forms
of the warped $\mathrm{G}_2$-structure on $M=I_f \times L$ in terms of the torsion forms of the $\mathrm{SU}(3)$-structure on $L$,
the warping function $f$,
and the functions $\alpha, \beta$.
For the sake of simplicity, in the next results we will not write the $t$-dependence of the functions $f, \alpha$ and~$\beta$.

The following lemma will be useful to relate the Hodge star operators $\ast_6$ and $\ast_7$ induced
by the $\mathrm{SU}(3)$ and $\mathrm{G}_2$ structures, respectively.

\begin{lemma}\label{lemma1}
Let $\gamma \in \Omega^p(L)$ be a differential $p$-form on $L$, and let $\ast_6$ and $\ast_7$ be the Hodge star operators induced by
the structures $(\omega, \psi_+, \psi_-)$ and $\varphi$, respectively. Then,
$$
\ast_7\gamma=f^{6-2p} \ast_6\!\gamma \wedge dt,\quad\quad
\ast_7(\gamma \wedge dt)=(-1)^pf^{6-2p}\ast_6\!\gamma.
$$
\end{lemma}

\begin{proof}
It is an immediate consequence of the definition of the Hodge star operator and the fact that $\ast_6$ and $\ast_7$ are determined,
respectively, by
$(g_{\omega,\psi_+}, vol_6=\frac{1}{6} \omega^3)$ and $(g_{\varphi}, vol_7)$, with $vol_7=f^6 vol_6\wedge dt$.
\end{proof}

\begin{proposition}\label{proposition1}
Let $\varphi$ be a warped $\mathrm{G}_2$-structure on $M= I_f \times L$.
Then,
\begin{equation*}\label{varphi}
\begin{aligned}
d  \varphi  = &  \, -f^2 \Big( \frac{3}{2} \sigma_0 + 3f' \alpha + f \alpha' \Big) \psi_+ \wedge dt  +f^2 \Big( \frac{3}{2} \pi_0 + 3f' \beta + f \beta' \Big) \psi_- \wedge dt \\
& \, +f^3 \big( \alpha\, \pi_0 - \beta\, \sigma_0 \big) \omega^2 + f^2 \nu_1 \wedge \omega \wedge dt + f^2 \nu_3 \wedge dt\\
& \, +f^3 \pi_1 \wedge (\alpha\, \psi_+ - \beta\, \psi_-) -f^3 \big( \alpha\, \pi_2 -  \beta\, \sigma_2 \big) \wedge \omega, \\[4pt]
d \ast_{7} \varphi  =  & \, f^3 \big(  2f' + \beta\, \pi_0 + \alpha\, \sigma_0 \big) \omega^2 \wedge dt + f^4 \nu_1 \wedge \omega^2 \\
& \, -f^3 \pi_1 \wedge \big( \beta\, \psi_+ - \alpha\, \psi_- \big) \wedge dt + f^3 \big( \beta\, \pi_2 + \alpha\, \sigma_2 \big) \wedge \omega \wedge dt,
\end{aligned}
\end{equation*}
where we denote by $\pi_0, \sigma_0, \pi_1, \nu_1, \pi_2, \sigma_2$ and $\nu_3$ the torsion forms
of the $\mathrm{SU}(3)$-structure $(\omega, \psi_+, \psi_+)$ on $L$.
\end{proposition}

\begin{proof}
For $d\varphi$, the result is a direct consequence of equations \eqref{SU3torsion} and Proposition \ref{prop-3-form}.
On the other hand, from Lemma~\ref{lemma1} it follows that
\begin{equation*}
\ast_{7} \varphi= \frac{1}{2} f^4\, \omega \wedge \omega + f^3\big( \beta\, \psi_+ + \alpha\, \psi_-  \big) \wedge dt,
\end{equation*}
and the result for $d\ast_7 \varphi$ is obtained also as a direct consequence of~\eqref{SU3torsion} and Proposition~\ref{prop-3-form}.
\end{proof}

\begin{theorem}\label{torsiones}
Let $(L, \omega, \psi_+, \psi_-)$ be an $\mathrm{SU}(3)$ manifold with torsion forms $\pi_0, \sigma_0, \pi_1, \nu_1, \pi_2, \sigma_2$ and $\nu_3$.
Then, the torsion forms of a warped $\mathrm{G}_2$ manifold $(M=I_f \times L, \varphi)$ are given by
\begin{equation*}\label{torsion}
\begin{aligned}
\tau_0   = &  \, \frac{4}{7f} \big(3\,\pi_0\,\alpha -3\,\sigma_0\,\beta + f \alpha \beta' - f \beta \alpha' \big),\\[4pt]
\tau_1  = &  \, \frac{1}{2f} \big(\pi_0\,\beta + \sigma_0\,\alpha +2f'\big) dt + \frac{\nu_1}{6}+\frac{\pi_1}{6}, \\[4pt]
\tau_2  = &  \, - \frac23\ast_6\!(\nu_1\wedge \omega^2)\wedge dt+\frac13\ast_6\!(\pi_1\wedge\omega^2)\wedge dt
-\frac13 f \beta \ast_6\!(\pi_1\wedge \psi_+) -\frac13 f \alpha \ast_6\!(\pi_1\wedge \psi_-)\\
& \, +\frac23 f \beta \ast_6\!(\nu_1\wedge \psi_+)
+\frac23 f \alpha \ast_6\!(\nu_1\wedge \psi_-)
-f \beta\, \pi_2 - f \alpha\, \sigma_2, \\[4pt]
\tau_3 = & \, - \frac{3}{14} f^2 \left(\pi_0\,\alpha^2 - \sigma_0\,\alpha\beta - 2 f \beta' \right)\psi_+
+ \frac{3}{14} f^2\left(\pi_0\,\alpha\beta- \sigma_0\,\beta^2 +2f \alpha'\right)\psi_-\\
& \, + \frac{2}{7} f\left( \pi_0\,\alpha - \sigma_0\,\beta -2f\alpha \beta'+2f\beta\alpha' \right)\omega \wedge dt
-\frac12 \ast_6\!(\nu_1\wedge\omega)+\frac12 \ast_6\!(\pi_1\wedge\omega)\\
& \, +\frac12 f \alpha \ast_6\!(\pi_1\wedge\psi_+)\wedge dt
-\frac12 f \beta \ast_6\!(\pi_1\wedge\psi_-)\wedge dt -\frac12 f \alpha \ast_6\!(\nu_1\wedge\psi_+)\wedge dt \\
& \, +\frac12 f \beta \ast_6\!(\nu_1\wedge\psi_-)\wedge dt + f (\alpha\, \pi_2- \beta\, \sigma_2) \wedge dt - f^2\ast_6\!\nu_3. \\
           \end{aligned}
\end{equation*}
\end{theorem}

\begin{proof}
From \eqref{torsionforms} it can be easily obtained that
\begin{equation*}
\begin{array}{lll}
& \tau_0  = \frac{1}{7} \ast_7\!(d\, \varphi \wedge \varphi),\quad
& \tau_2  = -\ast_7 d \ast_7 \varphi + 4\ast_7\!(\tau_1 \wedge \ast_7 \varphi), \\[5pt]
& \tau_1  = -\frac{1}{12} \ast_7\!(\ast_7 d \, \varphi \wedge \varphi), \quad
& \tau_3 = \ast_7 d \varphi -\tau_0\, \varphi -3 \ast_7\!(\tau_1 \wedge \varphi).
\end{array}
\end{equation*}
Let us detail the computations for $\tau_0$. By Proposition~\ref{proposition1} we have
\begin{equation*}
\begin{aligned}
d\varphi \wedge \varphi = & \, \Big[\!-f^2 \Big( \frac{3}{2} \sigma_0 + 3f' \alpha + f \alpha' \Big) \psi_+ \wedge dt
+f^2 \Big( \frac{3}{2} \pi_0 + 3f' \beta + f \beta' \Big) \psi_- \wedge dt \\
                                & \, +f^3 \big( \pi_0\, \alpha - \sigma_0\, \beta \big)\, \omega^2 + f^2 \nu_1 \wedge \omega \wedge dt + f^2 \nu_3 \wedge dt\\
                                & \, +f^3 \pi_1 \wedge (\alpha\, \psi_+ - \beta\, \psi_-) -f^3 \big( \alpha\, \pi_2 - \beta\, \sigma_2 \big) \wedge \omega\Big] \wedge \Big[ f^2 \omega \wedge dt + f^3\big( \alpha\, \psi_+ - \beta\, \psi_-  \big) \Big]\\
                                  = & \, f^5 ( \pi_0\, \alpha - \sigma_0\, \beta )\, \omega^3 \wedge dt
                                + \alpha f^5 \Big( \frac32 \pi_0 + 3f' \beta + f \beta' \Big) \psi_+ \wedge \psi_- \wedge dt \\
                                & \,   - \beta f^5 \Big( \frac32 \sigma_0 + 3f' \alpha + f \alpha' \Big) \psi_+ \wedge \psi_- \wedge dt \\
                                  = & \, f^5 ( \pi_0\, \alpha - \sigma_0\, \beta )\, \omega^3 \wedge dt
                                  + f^5 \big( \pi_0\, \alpha +\frac23 f \alpha \beta' - \sigma_0\, \beta - \frac{2}{3} f \beta \alpha' \big)\, \omega^3 \wedge dt \\
                                 = & \, f^5 \big( 2\pi_0\, \alpha -2 \sigma_0\, \beta + \frac23 f \alpha \beta' -\frac23 f \beta \alpha' \big) \, \omega^3 \wedge dt.  %%\\
                                 %%= & \, f^5 \big( 12\pi_0\, \alpha -12 \sigma_0\, \beta + 4 \alpha f \beta' -4 \beta f \alpha' \big) e^{123456} \wedge dt.
                                   \end{aligned}
\end{equation*}
Therefore, using Lemma~\ref{lemma1} we get
\begin{equation*}
\tau_0 = \frac{1}{7} \ast_7\!(d\, \varphi \wedge \varphi)
= \, \frac{4}{7f} \big(3\pi_0\, \alpha -3\sigma_0\, \beta + f \alpha \beta' - f \beta \alpha' \big).
\end{equation*}

Similarly, the results for $\tau_1, \tau_2$ and $\tau_3$ follow as a long but standard computation
taking into account Proposition~\ref{proposition1} and Lemmas~\ref{lemmaprop} and~\ref{lemma1}.
\end{proof}

An immediate consequence of the previous theorem is the following

\begin{corollary}\label{corolary1}
The torsion forms of a warped $\mathrm{G}_2$-structure satisfy:
\begin{equation*}
\begin{aligned}
\tau_0   = 0  & \iff \left \{  \begin{array}{ll} i) \, \, \, \, \,  \, \, & 3\pi_0\, \alpha -3 \sigma_0\, \beta+ f \alpha \beta' - f \beta \alpha'=0;
\end{array}
\right.
\\
\tau_1  = 0  & \iff    \left \{  \begin{array}{ll} ii)   \, \, \, & \sigma_0\, \alpha + \pi_0\, \beta + 2f'=0,\\[3pt]
                                       iii)  \,  \, \, &\pi_1 = -\nu_1;
\end{array}
\right.
\\
\tau_2  = 0  & \iff \left \{  \begin{array}{ll} iv)  \, \, \, \, & \pi_1 = 2 \nu_1,\\[3pt]
                                       v)  \,  \, \, \, &\beta \pi_2 + \alpha \sigma_2 =0;
\end{array}
\right.
  \\
 \tau_3 = 0 & \iff \left \{  \begin{array}{ll} vi)   & \pi_0\,\alpha -  \sigma_0\,\beta -2 f \alpha \beta'  +2 f \beta \alpha' =0, \\[3pt]
 vii)   & \pi_1= \nu_1, \\[3pt]
 viii) & \alpha \pi_2 - \beta \sigma_2 =0,\\[3pt]
 ix)  & \nu_3=0.
\end{array}
\right.
 \\
\end{aligned}
\end{equation*}
\end{corollary}

\begin{proof}
The result is obvious for $\tau_0, \tau_1$ and $\tau_2$ in view of Theorem~\ref{torsiones}.
For $\tau_3$, the vanishing of the first three summands (see Theorem~\ref{torsiones}) is equivalent to $vi)$. Indeed,
\begin{equation*}
\pi_0\, \alpha^2 - \sigma_0\, \alpha \beta - 2f \beta'  = \, \alpha \Big(\pi_0\,\alpha - \sigma_0\, \beta -2 f \alpha \beta'  +2 f \beta \alpha'\Big)
\end{equation*}
and
\begin{equation*}
\pi_0\, \alpha \beta - \sigma_0\, \beta^2 + 2 f \alpha' =  \, \beta \Big(\pi_0\,\alpha - \sigma_0\, \beta -2 f \alpha \beta'  +2 f \beta \alpha'\Big),
\end{equation*}
where we are using the fact that $\alpha \alpha'=-\beta \beta'$, which follows from the identity
$\alpha^2+\beta^2=1$. The other conditions $vii), viii)$ and $ix)$ are clear from Theorem~\ref{torsiones}.
\end{proof}

\end{section}

%%%%%%%%%%%%%%%%%%%%%%%%%%%%%%%%%%%%%%%%%%%%%%%%%%%%%%
%%%%%%%%%%%%%%%%%%%%%%%%%%%%%%%%%%%%%%%%%%%%%%%%%%%%%%

\begin{section}{Einstein warped $\mathrm{G}_2$ manifolds}\label{sec-Einstein-G2}

\noindent Our goal in this section is to construct Einstein $7$-manifolds in the different $\mathrm{G}_2$-classes
by means of warped products of certain Einstein $\mathrm{SU}(3)$ manifolds.
The $\mathrm{G}_2$-structures 
are of the form \eqref{3-form}, i.e. what we called warped $\mathrm{G}_2$-structures. In this way we will obtain explicit Einstein examples with scalar curvature of different signs. In Section~\ref{sec-main-classes} we study the principal classes of $\mathrm{G}_2$ manifolds,
Section~\ref{sec-coclosed} is devoted to 
coclosed $\mathrm{G}_2$-structures,
in Section~\ref{sec-coupled} warped products of 
coupled $\mathrm{SU}(3)$-structures are considered, and
in Section~\ref{sec-solvmanifolds} we obtain 
$\mathrm{G}_2$ structures on
the hyperbolic cosine cone of Einstein solvmanifolds.

Let us consider the warped product $M = B \times_f F$, i.e. the product manifold $B \times F$
endowed with the metric $g$ given by
$g = f^2 q^*(g_F) +  p^*(g_B)$,
with $p$ and $q$ the projections of $B \times F$ onto $B$ and $F$, respectively, and $f$ a nowhere vanishing
differentiable function on $B$.
We denote by $Ric^B$ the lift to $M$ (i.e. the pullback by $p$) of the Ricci curvature of $B$,
similarly for $Ric^F$, and let $Hess(f)$ be the lift to $M$ of the Hessian of  $f$.
By \cite[p. 211]{ON} the warped product  $M = B \times _f F$  is Einstein with constant $\lambda$ (i.e. $Ric =\lambda\, g$)
if and only if $(F, g_F)$ is Einstein with constant $\mu$ (i.e. $Ric^F = \mu\, g_F)$ and the following conditions are satisfied:
$$
\lambda\, g_B =  Ric^B - \frac{d}{f} Hess (f),\quad\quad
\lambda =  \frac{\mu}{f^2} -  \frac{\Delta f}{f} - (d -1) \left | \frac{\nabla f } {f}  \right |_{g_B}^2,
$$
where $d=\dim F \geq 2$, $\Delta f ={\rm tr}\, \big(Hess(f)\big)$, and $\nabla f$ denotes the gradient of $f$.

Moreover, when the base space $B$ has dimension 1, these equations %\eqref{Einstein}
reduce to
\begin{equation}\label{q=1}
(f')^2+\frac{\lambda}{d}f^2=\frac{\mu}{d-1}.
\end{equation}
The behavior of the solutions of \eqref{q=1} depends on the signs of the Einstein constants $\lambda$ and $\mu$.
Nevertheless, up to homotheties, those solutions (besides the constant case) are given in
Table~\ref{tableBesse} (see also \cite{Besse}).

\begin{table}[h]
\caption{\textbf{Solutions of the equation \eqref{q=1}}}
\begin{center}
\begin{tabular}{|c|c|c|c|c|c|}\hline
\hspace{0.2 cm} $\mu$ \hspace{0.2 cm} & \hspace{0.2 cm} $-(d-1)$ \hspace{0.2 cm} & \hspace{0.2 cm} $0$ \hspace{0.2 cm} & \hspace{0.2 cm} $d-1$ \hspace{0.2 cm} & \hspace{0.2 cm} $d-1$ \hspace{0.2 cm} & \hspace{0.2 cm} $d-1$ \hspace{0.2 cm} \\\hline
\hspace{0.2 cm} $\lambda$ \hspace{0.2 cm} & \hspace{0.2 cm} $-d$ \hspace{0.2 cm} & $-d$ & $-d$ &$0$&$d$ \\\hline \hline
\hspace{0.2 cm} $f(t)$ \hspace{0.2 cm} & $\cosh t$ & $e^t$ & $\sinh t$ &$t$&$\sin t$ \\\hline
\end{tabular}
\label{tableBesse}
\end{center}
\end{table}

From this table the next result  follows
\begin{theorem}\label{teorema-del-Besse}\cite[Theorem 9.110]{Besse}
Let $M = B \times_f F$ be a warped product, with $\dim B =1$ and $\dim F =d>1$.
If $M$ is a complete Einstein manifold, then either $M$ is a Ricci-flat Riemannian product, or
$B=\mathbb{R}$, $F$ is Einstein with non-positive scalar curvature and $M$ has negative scalar curvature.
\end{theorem}

We consider $B=I_f \subset \mathbb{R}$ an open interval
where the function $f(t)$ does not vanish.
For the functions in Table~\ref{tableBesse} we will take generically $I_f=\mathbb{R}$ for $f(t)=\cosh t$ or $e^t$, $I_f=(0,\infty)$ for $f(t)=\sinh t$ or $t$,
and $I_f=(0,\pi)$ for $f(t)=\sin t$. Following~\cite{FIMU}, for the latter case, if $F$ is compact then the
product manifold $[0,\pi] \times F$ is a compact manifold endowed with a metric
$g = dt^2 + \sin^2 t \, q^*(g_F)$ having two
conical singularities at $t=0$ and  $t=\pi$.

In order to use directly Table \ref{tableBesse}, we will consider the Einstein metric on the fiber $F$ to be ``normalized'',
that is, its Einstein constant is
$$-(d-1), \quad 0, \quad \text{ or } \quad d-1,$$
where $d$ denotes the dimension of $F$, or equivalently, the scalar curvature is
$$-d\, (d-1), \quad 0, \quad \text{ or } \quad d\, (d-1),$$
respectively.
There is no loss of generality in assuming this condition since every Einstein metric can be normalized via a rescaling.
Similar considerations are applied to Einstein metrics on the total space $M$ of the warped product.

\begin{subsection}{Principal classes of $\mathrm{G}_2$ manifolds}\label{sec-main-classes}

\noindent In this section we focus on Einstein $7$-manifolds in the principal classes
of $\mathrm{G}_2$ manifolds, i.e. in the classes $\mathcal{P}$, $\mathcal{X}_1$, $\mathcal{X}_2$, $\mathcal{X}_3$
and $\mathcal{X}_4$. Whereas one can construct Einstein manifolds in the classes $\mathcal{P}$, $\mathcal{X}_1$
and $\mathcal{X}_4$ by means of warped $\mathrm{G}_2$-structures,
however we will prove in Proposition~\ref{2+3-G2} that such a manifold in the class
$\mathcal{X}_2 \oplus \mathcal{X}_3$ is necessarily parallel.

Next, several characterizations will be given for the classes $\mathcal{P}$, $\mathcal{X}_1$ and $\mathcal{X}_4$.
We begin with parallel $\mathrm{G}_2$ manifolds.

\begin{proposition}\label{sec-P}
There exists a parallel  warped $\mathrm{G}_2$-structure on $M= I_f \times L$
if and only if
the fiber $(L, \omega, \psi_+, \psi_-)$ belongs to $\mathcal{W}_1^+ \oplus \mathcal{W}_1^-$ and is Einstein with $Scal(g_{\omega, \psi_+})=30$.

Furthermore, in that case $M=(0,\infty) \times L$ is the $t$-cone with the $\mathrm{G}_2$-structure
\begin{equation}\label{t-cone-P}
\varphi= t^2\, \omega \wedge dt + t^3 \left(-\frac{\sigma_0}{2} \psi_+ + \frac{\pi_0}{2}\psi_-\right),
\end{equation}
where $\sigma_0,\pi_0$ are the (constant) torsion functions of the $\mathrm{SU}(3)$-structure, which satisfy $\pi_0^2+\sigma_0^2 = 4$.
\end{proposition}

\begin{proof}
Let us suppose that the $\mathrm{SU}(3)$ manifold $(L, \omega, \psi_+, \psi_-)$ belongs to $\mathcal{W}_1^+ \oplus \mathcal{W}_1^-$ and is Einstein with constant~5.
Hence, the torsion reduces to $\pi_0$ and $\sigma_0$, and
the equations \eqref{SU3torsion} are given by
$$
d\omega  = -\frac{3}{2}\sigma_0\, \psi_++\frac{3}{2}\pi_0\, \psi_-, \quad\quad
d\psi_+  = \pi_0\, \omega^2, \quad\quad
d\psi_-  = \sigma_0\, \omega^2.
$$
These equations imply that the wedge product of the 1-forms $d\pi_0,d\sigma_0$ by $\omega^2$ is zero, so
$\pi_0,\sigma_0$ are constant.
Moreover, from \eqref{scal6} we get
$30 = Scal(g_{\omega, \psi_+}) = \frac{15}{2} (\pi_0^2+  \sigma_0^2)$,
which implies $\pi_0^2+\sigma_0^2 = 4$. Now, the
warped $\mathrm{G}_2$-structure  with $f(t)=t$, $\alpha=-\frac{\sigma_0}{2}$ and  $\beta=-\frac{\pi_0}{2}$
satisfies the equations $i)-ix)$ in Corollary~\ref{corolary1}, so it is parallel.

Conversely, let us suppose that there exists a warped $\mathrm{G}_2$-structure that is parallel,
i.e. the equations $i)-ix)$ in Corollary~\ref{corolary1} are satisfied. From \emph{iii)}, \emph{iv)} and \emph{ix)}
we have that $\pi_1=\nu_1=\nu_3=0$,
and from \emph{v)} and \emph{viii)} we get $\sigma_2=\pi_2=0$ because $\alpha^2(t)+\beta^2(t)=1$. Hence,
the manifold $(L,\omega,\psi_+,\psi_-)$ belongs to the $\mathrm{SU}(3)$-class $\mathcal{W}_1^+ \oplus \mathcal{W}_1^-$,
and by the first part of the proof we have that the torsion functions $\pi_0$ and $\sigma_0$ are constant.
Furthermore, by \eqref{scal7} any $\mathrm{G}_2$-parallel structure is Ricci-flat, so from Table~\ref{tableBesse}
we get that the warping function is necessarily $f(t)=t$ and the metric induced by the $\mathrm{SU}(3)$-structure
is Einstein with constant $\mu=5$. Notice that \eqref{scal6} implies $\pi_0^2+\sigma_0^2 = 4$.

Finally, it remains to see that the $\mathrm{G}_2$-structure on the $t$-cone is given by~\eqref{t-cone-P}.
Let us write $\alpha(t)=\cos\theta(t)$ and $\beta(t) =\sin\theta(t)$,
for some function $\theta(t)$. The equations $i)$ and $vi)$ for $f(t)=t$ are equivalent to
$$
\pi_0\, \alpha(t) - \sigma_0\, \beta(t) = 0, \qquad \theta'(t) = 0,
$$
which implies that $\alpha(t), \beta(t)$ are constant functions.
On the other hand, from the first equation above and the equation $ii)$ for $f(t)=t$,
we arrive at the following system
$$
\pi_0\, \alpha - \sigma_0\, \beta = 0,\qquad  \alpha_0\, \alpha + \pi_0\, \beta  = -2.
$$
Now, the condition $\pi_0^2+\sigma_0^2 = 4$ clearly implies that
$\alpha = -\frac{\sigma_0}{2}$ and $\beta = -\frac{\pi_0}{2}$,
and the result follows.
\end{proof}

In the following proposition we consider warped $\mathrm{G}_2$ manifolds in the class $\mathcal{X}_1$.
The result also gives another characterization of an $\mathrm{SU}(3)$ manifold
in the class $\mathcal{W}_1^+ \oplus \mathcal{W}_1^-$ in terms of a $\sin t$-cone.

\begin{proposition}\label{sec-X1}
There exists a nearly parallel warped $\mathrm{G}_2$-structure on $M=I_f \times L$ with $Scal(g_{\varphi})=42$
if and only if
the fiber $(L,\omega,\psi_+,\psi_-)$  belongs to $\mathcal{W}_1^+ \oplus \mathcal{W}_1^-$ and is Einstein with $Scal(g_{\omega, \psi_+})=30$.

Furthermore, in that case $M=(0,\pi) \times L$ is the $\sin t$-cone with the $\mathrm{G}_2$-structure
\begin{equation}\label{sine-t-cone-X1}
\varphi= \sin^2 t\, \omega \wedge dt + \sin^3 t \left(\cos (\varepsilon\,t+\rho)\, \psi_+ - \sin (\varepsilon\,t+\rho)\, \psi_-\right),
\end{equation}
where $\varepsilon= \pm 1$ and $\rho$ is given in terms of the (constant) torsion functions $\sigma_0,\pi_0$ of the $\mathrm{SU}(3)$-structure
by $\sigma_0=-2\,\cos \rho$ and $\pi_0=-2\,\sin \rho$.
\end{proposition}

\begin{proof}
Suppose that the $\mathrm{SU}(3)$ manifold belongs to $\mathcal{W}_1^+ \oplus \mathcal{W}_1^-$ and is Einstein with constant~5.
Hence, a similar argument as in the first part of the proof of Proposition~\ref{sec-P} shows that
$\pi_0,\sigma_0$ are constant and $\pi_0^2+\sigma_0^2 = 4$. Now, the
$\mathrm{G}_2$-structure given by \eqref{sine-t-cone-X1}
satisfies the equations $ii)-ix)$ in Corollary~\ref{corolary1}.
Thus, we get a nearly parallel $\mathrm{G}_2$ manifold with Einstein constant equal to $6$.

Let us prove the converse. Suppose that there exists a warped product of $(L,\omega,\psi_+,\psi_-)$ given by~\eqref{3-form} that is
a nearly parallel $\mathrm{G}_2$ manifold with Einstein constant $6$,
i.e. the equations $ii)-ix)$ in Corollary~\ref{corolary1} are satisfied.
The equations \emph{iii)}, \emph{iv)} and \emph{ix)}
imply $\pi_1=\nu_1=\nu_3=0$, and from \emph{v)} and \emph{viii)} we get $\sigma_2=\pi_2=0$ because $\alpha^2(t)+\beta^2(t)=1$.
On the other hand, by Table~\ref{tableBesse}
we get that the warping function is necessarily $f(t)=\sin t$ and the metric induced by the $\mathrm{SU}(3)$-structure
is Einstein with constant $\mu=5$, which implies, by \eqref{scal6}, that $\pi_0^2+\sigma_0^2 = 4$.
Hence,
the manifold $(L,\omega,\psi_+,\psi_-)$ belongs to the $\mathrm{SU}(3)$-class $\mathcal{W}_1^+ \oplus \mathcal{W}_1^-$,
and the (constant) torsion functions $\pi_0,\sigma_0$ satisfy $\pi_0^2+\sigma_0^2 = 4$.

It remains to prove that the warped product $M$ must be necessarily the $\sin t$-cone given in~\eqref{sine-t-cone-X1}.
To see this, we consider the equations $ii)$ and $vi)$ for $f(t)=\sin t$
in Corollary~\ref{corolary1}. Writing $\alpha(t)=\cos\theta(t)$ and $\beta(t) =\sin\theta(t)$,
for some function $\theta(t)$, we get
$$
\sigma_0\, \alpha(t) + \pi_0\, \beta(t) = -2\, \cos t, \qquad
\pi_0\, \alpha(t) - \sigma_0\, \beta(t) = 2\, \theta'(t)\, \sin t.
$$
Using $\pi_0^2+\sigma_0^2 = 4$, we have
$$
\alpha(t) = -\frac{1}{2}\, \sigma_0\, \cos t + \frac{1}{2}\, \pi_0\, \theta'(t) \sin t, \qquad
\beta(t) = -\frac{1}{2}\, \pi_0\, \cos t - \frac{1}{2}\, \sigma_0\, \theta'(t) \sin t,
$$
and from $\alpha^2(t)+\beta^2(t)=1$ it follows that
$$
\left[ \left(\theta'(t)\right)^2 -1 \right] \sin^2 t = 0.
$$
This implies $\theta'(t)= \pm 1$ and thus $\theta(t)=\varepsilon\,t+\rho$, where $\varepsilon=\pm 1$
and $\rho$ is a constant which, as we show next, it is determined by $\sigma_0$ and $\pi_0$.
Indeed, the equations $ii)$ and $vi)$ are now written as
$$
\begin{aligned}
& (\sigma_0\,\cos\rho + \pi_0\,\sin\rho + 2) \cos t + \varepsilon (\pi_0\,\cos\rho - \sigma_0\,\sin\rho) \sin t = 0, \\[4pt]
& (\pi_0\,\cos\rho - \sigma_0\,\sin\rho) \cos t - \varepsilon (\sigma_0\,\cos\rho + \pi_0\,\sin\rho + 2) \sin t = 0.
\end{aligned}
$$
These equations imply
$$
\sigma_0\,\cos\rho + \pi_0\,\sin\rho =- 2, \qquad  \sigma_0\,\sin\rho -\pi_0\,\cos\rho =0,
$$
whose solution is $\sigma_0=-2\,\cos\rho$ and $\pi_0=-2\,\sin\rho$.
In conclusion, the $\mathrm{G}_2$-structure is given by \eqref{sine-t-cone-X1}
and the proof is complete.
\end{proof}

\begin{corollary}\label{cor}
Let $(L,\omega,\psi_+,\psi_-)$ be an $\mathrm{SU}(3)$ manifold in $\mathcal{W}_1^+ \oplus \mathcal{W}_1^-$ with $Scal(g_{\omega, \psi_+})=30$.
Then, the nearly parallel $\mathrm{G}_2$-structure on $M=I_f \times L$ given by \eqref{sine-t-cone-X1}
has torsion $\tau_0= 4\, \varepsilon$ $(\varepsilon= \pm 1)$.
\begin{proof}
It is a direct consequence of Proposition~\ref{sec-X1} and the expression of $\tau_0$ in Theorem~\ref{torsiones}, taking $f(t)= \sin t$,
$\alpha(t)= \cos(\varepsilon t +\rho)$, $\beta(t)= \sin(\varepsilon t +\rho)$, $\cos \rho = -\frac{\sigma_0}{2}$ and $\sin \rho = -\frac{\pi_0}{2}$.
\end{proof}
\end{corollary}

As a consequence of Propositions~\ref{sec-P} and \ref{sec-X1} we recover well-known
characterizations of a nearly-K\"ahler manifold $L$ given in~\cite{Bar} and in~\cite{FIMU} (see also \cite{BoyerGalicki}).
Here, and in what follows, we consider that the torsion of a nearly-K\"ahler manifold is $\sigma_0=-2$, so the Einstein constant equals 5.

\begin{corollary}\label{NK-characterizations}
Let $(L,\omega,\psi_+)$ be an $\mathrm{SU}(3)$ manifold. Then:
\begin{enumerate}
\item[{\rm (i)}] $L$ is nearly-K\"ahler if and only if the (usual) cone with the $\mathrm{G}_2$-structure
$$
\varphi= t^2\, \omega \wedge dt + t^3 \, \psi_+,
$$
is a parallel $\mathrm{G}_2$ manifold;
%%%
\item[{\rm (ii)}] $L$ is nearly-K\"ahler if and only if the sine-cone with the $\mathrm{G}_2$-structure
$$
\varphi= \sin^2 t\, \omega \wedge dt + \sin^3 t \left(\cos t\, \psi_+ - \sin t\, \psi_-\right),
$$
is a nearly parallel $\mathrm{G}_2$ manifold.
\end{enumerate}
\end{corollary}

\begin{proof}
For (i), just take in \eqref{t-cone-P} the values $\sigma_0=-2$ and $\pi_0=0$.
For (ii) we take $\varepsilon=1$ in \eqref{sine-t-cone-X1} and $\rho=0$, because $-2=\sigma_0=-2\,\cos \rho$ and $0=\pi_0=-2\,\sin \rho$.
\end{proof}

Recall that $\mathrm{G}_2$ manifolds in the class $\mathcal{X}_2 \oplus \mathcal{X}_3$ are characterized
in terms of the torsion forms by the conditions $\tau_0=\tau_1=0$.

\begin{proposition}\label{2+3-G2}
A warped $\mathrm{G}_2$ manifold $M$ in the class
$\mathcal{X}_2 \oplus \mathcal{X}_3$ is Einstein if and only if it is a parallel $\mathrm{G}_2$ manifold.
\end{proposition}

\begin{proof}
From Corollary~\ref{corolary1}, if the $\mathrm{G}_2$-structure
belongs to the class $\mathcal{X}_2 \oplus \mathcal{X}_3$ then the conditions $i), ii)$ and $iii)$ are satisfied. In addition,
an Einstein $\mathrm{G}_2$ manifold with $\tau_0=\tau_1=0$ has non-positive
Einstein constant by \eqref{scal7}. If such constant is zero then the $\mathrm{G}_2$-structure is parallel.
So, in what follows we suppose that the Einstein constant is negative, which after scaling we consider to be $-6$, and so by Table~\ref{tableBesse} the possible functions are $f(t)=\cosh t$, $e^t$, or $\sinh t$.
Next we will prove that there is no solution in any of
these cases.

From $\alpha^2(t)+\beta^2(t)=1$ we can write
$\alpha(t) =\cos \theta(t)$ and $\beta(t) =\sin \theta(t)$,
for some real-valued function $\theta(t)$. Thus, $\alpha(t)\beta'(t)-\beta(t) \alpha'(t)=\theta'(t)$,
and equations $i)$ and $ii)$ in Corollary~\ref{corolary1} become:
\begin{enumerate}
\item[{\rm $i)$}] $3\,\pi_0\, \alpha(t) - 3\,\sigma_0\, \beta(t) +\theta'(t)\, f(t) = 0$,
\item[{\rm $ii)$}] $\sigma_0\,\alpha(t) + \pi_0\,\beta(t) + 2f'(t) = 0$.
\end{enumerate}

Multiplying $i)$ by $\alpha(t)$, $ii)$ by $3\beta(t)$, and summing the resulting equations, we get
$$
3\,\pi_0 = 3\,\pi_0 (\alpha^2(t)+\beta^2(t)) = -\theta'(t)\,\alpha(t)\,f(t) - 6\,\beta(t)\,f'(t).
$$
Since $\pi_0$ is a function on the fiber manifold $L$ and the right hand side of the equation only depends on $t$,
necessarily there exists a constant $C_1$ such that
\begin{equation}\label{uno}
\theta'(t)\,\alpha(t)\,f(t) + 6\,\beta(t)\,f'(t) = C_1.
\end{equation}

Now, multiplying $i)$ by $-\beta(t)$, $ii)$ by $3\alpha(t)$, and summing the resulting equations, we get
$$
3\,\sigma_0 = 3\,\sigma_0 (\alpha^2(t)+\beta^2(t)) = \theta'(t)\,\beta(t)\,f(t) - 6\,\alpha(t)\,f'(t).
$$
Hence, there exists a constant $C_2$ such that
\begin{equation}\label{dos}
\theta'(t)\,\beta(t)\,f(t) - 6\,\alpha(t)\,f'(t) = C_2.
\end{equation}

Taking the product of \eqref{dos} by $\alpha(t)$, the product of \eqref{uno} by $\beta(t)$, and subtracting
the equations, we get $6\,f'(t)=C_1\,\beta(t)-C_2\,\alpha(t)$.
In a similar way, taking the product of \eqref{dos} by $\beta(t)$, the product of \eqref{uno} by $\alpha(t)$, and summing
the equations, we get $\theta'(t)\,f(t)=C_1\,\alpha(t)+C_2\,\beta(t)$.
That is, we arrive at the following system:
\begin{equation}\label{1}
6\,f'(t)=C_1\,\beta(t)-C_2\,\alpha(t),
\end{equation}
\begin{equation}\label{2}
\theta'(t)\,f(t)=C_1\,\alpha(t)+C_2\,\beta(t).
\end{equation}

Taking the derivative of \eqref{2} and using \eqref{1} we get
$\theta''(t)\,f(t)+\theta'(t)\,f'(t) = C_1\,\alpha'(t)+C_2\,\beta'(t)= -\theta'(t) (C_1\,\beta(t)-C_2\,\alpha(t))= -6\, \theta'(t) \,f'(t)$,
that is
$$
\theta''(t)\,f(t)+7\,\theta'(t)\,f'(t)=0.
$$
Notice that $\theta'(t) =0$ implies that the functions $\alpha(t)$ and $\beta(t)$ are constant, and then equation $ii)$
cannot be solved for $f(t)=\cosh t, e^t$, or $\sinh t$. Therefore, $\theta'(t) \not=0$ and we can write the previous equation as
$$
\left( \ln \theta'(t) + 7 \, \ln f(t)  \right)'=0.
$$
Hence, there exists a positive constant $C_0$ such that
\begin{equation}\label{3}
\theta'(t)= C_0 \, f(t)^{-7}.
\end{equation}

On the other hand, taking the derivative of \eqref{1} and using \eqref{2} we get
$6\,f''(t)=C_1\,\beta'(t)-C_2\,\alpha'(t) = \theta'(t) (C_1\,\alpha(t)+C_2\,\beta(t))= (\theta'(t))^2 \,f(t)$,
that is
$$
6\,f''(t) = (\theta'(t))^2 \,f(t).
$$
Now, using \eqref{3}, we have $6\,f''(t) = C_0^2 \, f(t)^{-13}$, i.e.
$$
f(t)^{13} f''(t) = C_0^2/6,
$$
which never holds for the functions $f(t)=\cosh t, e^t$, or $\sinh t$. In conclusion, the system $i)-iii)$ is never satisfied.
\end{proof}

Since the class $\mathcal{X}_2 \oplus \mathcal{X}_3$ contains the class of closed and the class of coclosed of pure type
$\mathrm{G}_2$ manifolds, from Proposition~\ref{2+3-G2} we get

\begin{corollary}\label{no-existes-2-ni-3}
There does not exist any $\mathrm{SU}(3)$ manifold $(L,\omega,\psi_+,\psi_-)$ for which the warped $\mathrm{G}_2$ manifold
$M= I_f \times L$ is Einstein closed or coclosed of pure type,
unless it is parallel.
\end{corollary}

\begin{remark}\label{conjetura}
{\rm
As we recall in the introduction, it is an open question if an Einstein closed $\mathrm{G}_2$ manifold must be parallel.
Several authors have proved that this question has an affirmative answer in different particular situations:
for compact (and more generally, for $\ast$-Einstein) manifolds in \cite{CI1,CI},
for non-negative scalar curvature in \cite{Br}, and
for solvmanifolds with left invariant $\mathrm{G}_2$-structure in \cite{FFM}.
The corollary above shows that the answer is also affirmative in the class of warped $\mathrm{G}_2$ manifolds.
}
\end{remark}

Now, we turn our attention to Einstein locally conformal parallel $\mathrm{G}_2$ manifolds, i.e. Einstein manifolds
in the class $\mathcal{X}_4$.

\begin{proposition}\label{sec-X4-negative}
There exists an Einstein locally conformal parallel warped $\mathrm{G}_2$-structure on $M=I_f\times L$ with $Scal(g_{\varphi})=-42$
if and only the fiber $(L,\omega,\psi_+,\psi_-)$  is one of the following:

\smallskip

\noindent $\bullet$ $L$ is Calabi-Yau, and then $M=\mathbb{R}\times L$ is the exponential-cone with $\mathrm{G}_2$-structure
$\varphi= e^{2t} \omega \wedge dt + e^{3t} \psi_+$, or

\smallskip

\noindent $\bullet$ $L$ belongs to $\mathcal{W}_1^+ \oplus \mathcal{W}_1^-$ with $Scal(g_{\omega,\psi_+})=30$,
and then $M= (0,\infty)\times L$ is the hyperbolic sine-cone with $\mathrm{G}_2$-structure
$\varphi = \sinh^2 t\, \omega \wedge dt + \sinh^3 t \left(\varepsilon\frac{\sigma_0}{2} \, \psi_+ - \varepsilon\frac{\pi_0}{2} \, \psi_-\right)$,
where $\varepsilon=\pm 1$ and $\sigma_0,\pi_0$ are the (constant) torsion functions of the $\mathrm{SU}(3)$-structure,
which satisfy $\pi_0^2+\sigma_0^2 = 4$.
\end{proposition}

\begin{proof}
Suppose there is such a warped product. Using that $\tau_0=\tau_2=\tau_3=0$ and Corollary~\ref{corolary1},
similarly to the proof of Proposition~\ref{sec-P} we arrive at the fact that
$L$ belongs to $\mathcal{W}_1^+ \oplus \mathcal{W}_1^-$, so the torsion reduces to $\sigma_0,\pi_0$.
If they vanish then $L$ is Calabi-Yau and the warped product is the exponential-cone.
If the torsion of $L$ is non-zero then the scalar curvature of $L$ is equal to $30$ and $f(t)=\sinh t$.
The equations $i)$ and $vi)$ in Corollary~\ref{corolary1} give the solutions $(\alpha,\beta)=(\varepsilon\frac{\sigma_0}{2},\varepsilon\frac{\pi_0}{2})$,
where $\varepsilon=\pm 1$.
\end{proof}

Similarly to the previous proposition we have:

\begin{proposition}\label{sec-X4-zero-positive}
Let $(L,\omega,\psi_+,\psi_-)$ be an $\mathrm{SU}(3)$ manifold. Then:
\begin{enumerate}
\item[{\rm (i)}]
There exists a Ricci flat locally conformal parallel warped $\mathrm{G}_2$-structure on $M=I_f\times L$
if and only the fiber $L$ belongs to $\mathcal{W}_1^+ \oplus \mathcal{W}_1^-$,
and then $M= (0,\infty)\times L$ is the cone with $\mathrm{G}_2$-structure
$\varphi = t^2\, \omega \wedge dt + t^3 \left(\varepsilon\frac{\sigma_0}{2} \, \psi_+ - \varepsilon\frac{\pi_0}{2} \, \psi_-\right)$,
where $\varepsilon=\pm 1$.  In addition, $M$ is parallel if and only if $\varepsilon=- 1$;
\item[{\rm (ii)}]
There exists an Einstein locally conformal parallel warped $\mathrm{G}_2$-structure on $M=I_f\times L$ with $Scal(g_{\varphi})=42$
if and only the fiber $L$ belongs to $\mathcal{W}_1^+ \oplus \mathcal{W}_1^-$,
and then $M= (0,\pi)\times L$ is the $\sin t$-cone with $\mathrm{G}_2$-structure
$\varphi = \sin^2 t\, \omega \wedge dt + \sin^3 t \left(\varepsilon\frac{\sigma_0}{2} \, \psi_+ - \varepsilon\frac{\pi_0}{2} \, \psi_-\right)$,
where $\varepsilon=\pm 1$.
\end{enumerate}
\end{proposition}

\end{subsection}

\begin{subsection}{Einstein coclosed $\mathrm{G}_2$ manifolds}\label{sec-coclosed}

In this section we construct Einstein coclosed $\mathrm{G}_2$-structures
(i.e. of type $\mathcal{X}_1 \oplus \mathcal{X}_3$) on warped products of
$\mathrm{SU}(3)$ manifolds in the class
$\mathcal{W}_1^+ \oplus \mathcal{W}_1^- \oplus \mathcal{W}_3$.
We apply the construction to the manifold $S^3 \times S^3$ endowed with one of the $\mathrm{SU}(3)$-structures described in \cite{Sc}.

\begin{theorem}\label{coclosed-einstein}
Let $(L, \omega, \psi_+, \psi_-)$ be an Einstein $\mathrm{SU}(3)$-structure of type $\mathcal{W}_1^+ \oplus \mathcal{W}_1^- \oplus \mathcal{W}_3$
with $Scal(g_{\omega,\psi_+})=30$.
Then, the torsion functions $\pi_0,\sigma_0$ are constant, and $C=\sqrt{\pi_0^2+\sigma_0^2}$ satisfies $C\geq 2$.

Moreover, let 
$a=\arccos (\sigma_0/C)$
and consider 
$\theta(t)$ as follows:
\begin{enumerate}
\item[{\rm (i)}] if $\theta(t)$ is the constant function $\theta = a- \arccos (-2/C)$, then
the $\mathrm{G}_2$-structure
$$
\varphi = t^2\, \omega \wedge dt + t^3 \Big(\cos \theta\,\psi_+ - \sin \theta\,\psi_-\Big)
$$
on the manifold $M=(0,\infty) \times L$ is coclosed and its induced metric is Ricci flat;
\item[{\rm (ii)}] if $\theta(t) = a- \arccos (-2 \cos t /C)$, then
the $\mathrm{G}_2$-structure
$$
\varphi = \sin^2t\, \omega \wedge dt + \sin^3t \Big(\cos \theta(t)\,\psi_+ - \sin \theta(t)\,\psi_-\Big)
$$
on the manifold $M=(0,\pi) \times L$ is coclosed
and its induced metric is Einstein with $Scal(g_{\varphi})=42$;
\item[{\rm (iii)}] if $C>2$ and $\theta(t) = a- \arccos (-2 \cosh t /C)$, then
the $\mathrm{G}_2$-structure
$$
\varphi = \sinh^2t\, \omega \wedge dt + \sinh^3t \Big(\cos \theta(t)\,\psi_+ - \sin \theta(t)\,\psi_-\Big)
$$
on the manifold $M=\left( 0,\ln \frac{C+\sqrt{C^2-4}}{2} \right) \times L$ is coclosed,
and its induced metric is Einstein with $Scal(g_{\varphi})=-42$.
\end{enumerate}
\end{theorem}

\begin{proof}
Since the $\mathrm{SU}(3)$-structure is of type $\mathcal{W}_1^+ \oplus \mathcal{W}_1^- \oplus \mathcal{W}_3$,
we have that the possibly non-zero torsion reduces to $\pi_0,\sigma_0$ and $\nu_3$, that is,
the equations \eqref{SU3torsion} reduce to
$$
d\omega  = -\frac{3}{2}\sigma_0\, \psi_++\frac{3}{2}\pi_0\, \psi_- + \nu_3, \quad\quad
d\psi_+  = \pi_0\, \omega^2, \quad\quad
d\psi_-  = \sigma_0\, \omega^2.
$$
These equations imply $d\pi_0\wedge \omega^2=0$ and $d\sigma_0\wedge \omega^2=0$, therefore the torsion functions
$\pi_0,\sigma_0$ are constant.

On the other hand,
from the expression \eqref{scal6} for the scalar curvature we get
$$
30 = Scal(g_{\omega, \psi_+}) = \frac{15}{2} (\pi_0^2+  \sigma_0^2) - \frac{1}{2} |\nu_3|^2
\leq \frac{15}{2} (\pi_0^2+  \sigma_0^2),
$$
which implies $C^2=\pi_0^2+\sigma_0^2\geq 4$.

Clearly, the $\mathrm{G}_2$-structure given by \eqref{3-form} has torsion form $\tau_2 = 0$.
Thus, it is coclosed if and only if $\tau_1 = 0$ or, equivalently by Corollary~\ref{corolary1},
if and only if the equation
$$
\sigma_0\, \alpha(t) + \pi_0\, \beta(t) = -2f'(t)
$$
is satisfied. The scalar curvature of $g_{\omega, \psi_+}$ is positive,
so  $f(t)$ must be $t$, $\sin t$ or $\sinh t$.

Let $a=\arccos (\sigma_0/C)$, i.e. $\sigma_0=C\cos a$ and $\pi_0=C\sin a$.
Writing $\alpha(t)=\cos \theta(t)$ and $\beta(t)=\sin \theta(t)$, the equation
above becomes $\sigma_0\, \alpha(t) + \pi_0\, \beta(t)= C\cos (a-\theta(t))=-2f'(t)$, that is,
$$
\theta(t) = a- \arccos (-2 f'(t) /C).
$$
For $f(t)=t$ or $\sin t$, we have that $|-2 f'(t) /C|\leq 1$ for any $t$, because $C\geq 2$.
However, for $f(t)=\sinh t$, since $\cosh t\geq 1$ we need to impose that $C> 2$ in order to get
an open interval of values of $t$ satisfying $|-2 \cosh t /C| < 1$. Indeed, such interval is
$(\ln \frac{C-\sqrt{C^2-4}}{2}, \ln \frac{C+\sqrt{C^2-4}}{2})$ when $C> 2$.
From this discussion, the cases (i), (ii) and (iii) follow directly.
\end{proof}

\begin{example}\label{example1}
{\rm
We will apply Theorem~\ref{coclosed-einstein} to an Einstein $\mathrm{SU}(3)$-structure on $S^3\times S^3$
in the class $\mathcal{W}_1^- \oplus \mathcal{W}_3$
found in~\cite{Sc}. Here we will follow the description given in \cite[Section 3.4]{Ma}.

Let us consider the sphere $S^3$, viewed as the Lie group $\mathrm{SU}(2)$, with the basis of left invariant 1-forms $\{e^1, e^2, e^3\}$ satisfying
\begin{equation*}
de^1=e^{23}, \quad de^2=-e^{13}, \quad \text{and} \quad de^3=e^{12}.
\end{equation*}
Hence, the Lie algebra of $S^3\times S^3$
is $\mathfrak{g}=\mathfrak{su}(2)\oplus\mathfrak{su}(2)$, and its structure equations are
\begin{equation*}
\mathfrak{g}=(e^{23}, -e^{13}, e^{12}, f^{23}, -f^{13}, f^{12}),
\end{equation*}
where $\{f^i\}$ denotes the basis of 1-forms on the second sphere.
Now, we consider the basis
$\{h^1, \dots, h^6\}$ of the dual space $\mathfrak{g}^*$ of $\mathfrak{g}$ given by
\begin{equation*}
h^1=\frac{\sqrt{5}}{10}(e^1+f^1), \, h^2=\frac{\sqrt{5}}{10}(-e^1+f^1), \,  h^3=\frac{\sqrt{10}}{10}\,e^2, \, h^4=\frac{\sqrt{10}}{10}\,f^2, \,  h^5=\frac{\sqrt{10}}{10}\,e^3, \,  h^6=\frac{\sqrt{10}}{10}\,f^3.
\end{equation*}
With respect to this basis, the structure equations of the Lie algebra $\mathfrak{g}$ of $S^3\times S^3$ turn into
\begin{equation*}
\mathfrak{g}= \left( \sqrt{5}(h^{35}+h^{46}), \sqrt{5}(-h^{35}+h^{46}),\sqrt{5}(-h^{15}+h^{25}),
\sqrt{5}(-h^{16}-h^{26}), \sqrt{5}(h^{13}-h^{23}), \sqrt{5}(h^{14}+h^{24}) \right).
\end{equation*}
We define the $\mathrm{SU}(3)$-structure $(\omega, \psi_+, \psi_-)$ on $S^3\times S^3$ by
$$
\omega = h^{12}+h^{34}+h^{56},\quad\quad
\psi_+ = h^{135}-h^{146}-h^{236}-h^{245},\quad\quad
\psi_- = h^{136}+h^{145}+h^{235}-h^{246}.
$$
Then, an easy calculation shows that the equations~\eqref{SU3torsion} are
\begin{equation*}\label{canonical-su(3)-double}
\begin{array}{lcl}
d\omega &\!\!\!=\!\!\!& -\frac32 \sigma_0\, \psi_+ +\nu_3,\\[4pt]
d\psi_+ &\!\!\!=\!\!\!& 0,\\[4pt]
d\psi_- &\!\!\!=\!\!\!& \sigma_0\, \omega\wedge \omega,
\end{array}
\end{equation*}
where $\sigma_0=-\sqrt{5}$ and the torsion form $\nu_3$ is given by
\begin{equation*} \label{expression:nu3}
\nu_3=-\frac{\sqrt{5}}{2}\,h^{135} + \frac{\sqrt{5}}{2}\,h^{146} - \frac{\sqrt{5}}{2}\,h^{236} - \frac{\sqrt{5}}{2}\,h^{245}
+ \sqrt{5}\,h^{235} + \sqrt{5}\,h^{246}.
\end{equation*}
Therefore, the $\mathrm{SU}(3)$-structure $(\omega, \psi_+, \psi_-)$ on $S^3\times S^3$
belongs to the class $\mathcal{W}_1^- \oplus \mathcal{W}_3$. Moreover,
the induced metric $g_{\omega, \psi_{+}}$ on $S^3\times S^3$ is given by
$g_{\omega, \psi_{+}}= \sum_{i=1}^6 h^i \otimes h^i$,
and its Ricci curvature tensor satisfies
\begin{equation*}\label{ricci:s3s3}
Ric(g_{\omega, \psi_{+}}) = 5\,g_{\omega, \psi_{+}}.  %%%\sum_{i=1}^6 h^i \otimes h^i.
\end{equation*}
Thus, $g_{\omega, \psi_{+}}$ is an Einstein metric on $S^3\times S^3$ with $Scal(g_{\omega,\psi_+})=30$.

We can apply Theorem~\ref{coclosed-einstein} to get Einstein coclosed $\mathrm{G}_2$ manifolds
with different scalar curvatures. Notice that $C=\sqrt{5}$ and $a=\pi$.
Thus, in case (i) we get $\alpha=\frac{2\sqrt{5}}{5}$ and $\beta=-\frac{\sqrt{5}}{5}$, that is,
the manifold $M= (0,\infty) \times S^3 \times S^3$ with the $\mathrm{G}_2$-structure
$$
\varphi= t^2 \omega \wedge dt + \frac{\sqrt{5}}{5} t^3  \Big(2\,\psi_+ - \psi_-\Big)
$$
is a Ricci flat coclosed $\mathrm{G}_2$ manifold.

In case (ii), we have that a slight modification of the sine-cone provides an Einstein coclosed $\mathrm{G}_2$ manifold.
More concretely, the $\mathrm{G}_2$-structure
$$
\varphi = \sin^2t\, \omega \wedge dt + \frac{\sqrt{5}}{5} \sin^3t  \Big( 2\cos t\, \psi_+ - \sqrt{5-4\cos^2 t}\, \psi_-\Big)
$$
on the manifold $M=(0,\pi) \times S^3 \times S^3$ is coclosed
and its induced metric is Einstein with positive scalar curvature.

Finally, since $C=\sqrt{5} > 2$ we can apply (iii) with
$\theta(t) = \pi- \arccos (-2\cosh t /\sqrt{5})$, to get that the $\mathrm{G}_2$-structure
$$
\varphi = \sinh^2t\, \omega \wedge dt + \frac{\sqrt{5}}{5} \sinh^3t \Big(2\cosh t\, \psi_+ - \sqrt{5-4\cosh^2 t}\, \psi_-\Big)
$$
on the manifold $M=\left( 0,\ln \frac{1+\sqrt{5}}{2} \right) \times S^3 \times S^3$ is coclosed
and its induced metric is Einstein with negative scalar curvature.
}
\end{example}

\end{subsection}

\begin{subsection}{Warped products of Einstein coupled manifolds}\label{sec-coupled}

In this section we consider warped products of $6$-manifolds endowed with a coupled $\mathrm{SU}(3)$-structure.
Coupled $\mathrm{SU}(3)$-structures were first introduced in~\cite{Sal-Milan} (see also~\cite{FR1} for their role in physics),
and they are characterized by the condition
\begin{equation}\label{coupled-eq}
d \omega = c\, \psi_+,
\end{equation}
where $c \in \mathbb{R}-\{0\}$ is a nonzero constant.
Equivalently, coupled $\mathrm{SU}(3)$-structures have torsion class $\mathcal{W}_1^-\oplus \mathcal{W}_2^-$, i.e. they are $\mathrm{SU}(3)$-structures for which all the torsion forms different from $\sigma_0$ and $\sigma_2$ vanish. Notice that the torsion function $\sigma_0$ is a constant such that $\sigma_0=-\frac{2\, c}{3}$.
Coupled $\mathrm{SU}(3)$-structures are half-flat and they generalize the nearly K\"ahler structures ($\sigma_2=0$).
The next result follows from Theorem~\ref{torsiones}.

\begin{proposition}\label{coupled}
Let $(M=I_f \times L, \varphi)$ a warped $\mathrm{G}_2$ manifold of a coupled $\mathrm{SU}(3)$ manifold
$(L, \omega, \psi_+,\psi_-)$. The torsion forms are
\begin{equation*}
\begin{aligned}
\tau_0   = &  \, - \frac{4}{7f} \big(3\, \beta\, \sigma_0 - f \alpha \beta' + f \beta \alpha' \big),\\
\tau_1  = &  \, \frac{1}{2f}(\alpha\, \sigma_0 +2f') dt, \\
\tau_2  = &  \,  - f \alpha\, \sigma_2, \\
\tau_3 = & \, \frac{3}{14} f^2 \left(\alpha\beta\, \sigma_0 + 2 f \beta' \right)\psi_+
- \frac{3}{14} f^2\left(\beta^2\sigma_0 - 2f \alpha'\right)\psi_-  \\
& \, - \frac{2}{7} f\left(\beta\,\sigma_0 + 2f\alpha \beta' - 2f\beta\alpha' \right)\omega \wedge dt
-f\beta\, \sigma_2 \wedge dt, \\
\end{aligned}
\end{equation*}
where $\sigma_0=-\frac{2}{3}\, c$.
\end{proposition}

Next we will consider coupled $\mathrm{SU}(3)$-structures with $\sigma_2\not=0$ (i.e. which are not nearly-K\"ahler,
since the latter case has been studied in Section~\ref{sec-main-classes})
which are Einstein with positive scalar curvature. In the following result
we restrict our attention to those warped $\mathrm{G}_2$-structures for which $\alpha$ and $\beta$ are constant.

\begin{theorem}\label{coupled-alpha-beta-constantes}
Let $(L, \omega, \psi_+,\psi_-)$ be a (non nearly-K\"ahler) Einstein coupled $\mathrm{SU}(3)$ manifold
with  $Scal(g_{\omega, \psi_+})=30$. Then, the coupled constant $c$ satisfies $|c|>3$, and
we have:
\begin{enumerate}
\item[{\rm (i)}]
If $(\alpha,\beta)=(1,0)$, then the $\mathrm{G}_2$-structure
$$
\varphi = f^2 \omega \wedge dt + f^3 \psi_+
$$
on the manifold $M=I_f \times L$ is locally conformal closed (i.e. of type $\mathcal{X}_2 \oplus \mathcal{X}_4$)
and its induced metric is Ricci flat for $f(t)=t$,
Einstein with $Scal(g_{\varphi})=42$ for $f(t)= \sin \, t$,
and Einstein with $Scal(g_{\varphi})=-42$ for $f(t)= \sinh \, t$.
\item[{\rm (ii)}]
If $(\alpha,\beta)=(0,1)$, then the $\mathrm{G}_2$-structure
$$
\varphi = f^2 \omega \wedge dt - f^3 \psi_-
$$
on the manifold $M=I_f \times L$ is integrable (i.e. of type $\mathcal{X}_1 \oplus \mathcal{X}_3 \oplus \mathcal{X}_4$)
and its induced metric is Ricci flat for $f(t)=t$,
Einstein with $Scal(g_{\varphi})=42$ for $f(t)= \sin \, t$,
and Einstein with  $Scal(g_{\varphi})=-42$ for $f(t)= \sinh \, t$.
\item[{\rm (iii)}]
If $(\alpha,\beta)=(\frac{3}{c}, \frac{\sqrt{c^2-9}}{c})$, then the $\mathrm{G}_2$-structure
$$
\varphi = t^2 \omega \wedge dt + \frac{t^3}{c} \Big(3\,\psi_+ - \sqrt{c^2-9}\,\psi_-\Big)
$$
on the manifold $M=(0,\infty) \times L$ is of type $\mathcal{X}_1 \oplus \mathcal{X}_2 \oplus \mathcal{X}_3$ with Ricci flat induced metric.
\end{enumerate}
\end{theorem}

\begin{proof}
Since $\sigma_0,\sigma_2$ do not vanish, from the expression \eqref{scal6}
for the scalar curvature we get
$$
30 = Scal(g_{\omega, \psi_+}) = \frac{15}{2} \sigma_0^2 - \frac{1}{2}|\sigma_2|^2 = \frac{15}{2} \left( -\frac{2}{3}\, c \right)^2 - \frac{1}{2}|\sigma_2|^2
< \frac{10}{3} \, c^2.
$$
Therefore, the coupled constant $c$ in~\eqref{coupled-eq} satisfies $c^2>9$.

Let $\alpha$ and $\beta$ be constant functions satisfying $\alpha^2+\beta^2=1$. Then, by Proposition~\ref{coupled}
the torsion forms of the warped $\mathrm{G}_2$-structure reduce to
\begin{equation*}
\begin{aligned}
\tau_0 = &  \,  \frac{8}{7f} \beta\, c ,\quad\ \
\tau_1 = \, \frac{1}{3f}(3f' - \alpha\, c) dt, \quad\ \
\tau_2 = \,  - f \alpha\, \sigma_2, \\[4pt]
\tau_3 = & \, - \frac{1}{7} f^2 \alpha\beta\, c \,\psi_+
+ \frac{1}{7} f^2 \beta^2 c \,\psi_-  + \frac{4}{21} f \beta\, c \,\omega \wedge dt
-f\beta\, \sigma_2 \wedge dt, \\
\end{aligned}
\end{equation*}
where $\beta= \pm \sqrt{1-\alpha^2}$ and $0 \leq |\alpha| \leq 1$.

In the case (i), since $\alpha=1$ and $\beta=0$ we get
$$
\tau_0 = 0,\quad \ \tau_1 = \frac{1}{3f}(3f' -  c) dt, \quad \ \tau_2 = - f \sigma_2, \quad \ \tau_3 = 0.
$$
Hence the torsion forms $\tau_0$ and $\tau_3$ vanish, i.e. the $\mathrm{G}_2$ manifold is locally conformal closed.
Applying Table~\ref{tableBesse} to the function $f(t)=t$ we have that the induced metric is Ricci flat,
and for the function $f(t)= \sin \, t$ (resp. $f(t)= \sinh \, t$)
the metric induced by the $\mathrm{G}_2$-structure is Einstein
with $Scal(g_{\varphi})=42$ (resp. $Scal(g_{\varphi})=-42$).

In the case (ii), since $\alpha=0$ and $\beta=1$ we have
$$
\tau_0   =  \frac{8}{7f}\, c ,\quad
\tau_1  =  \frac{f'}{f} dt, \quad
\tau_2  =  0, \quad
\tau_3 =  \frac{1}{7} f^2 c \,\psi_-  + \frac{4}{21} f \, c \,\omega \wedge dt - f\, \sigma_2 \wedge dt.
$$
Since $\tau_2=0$, the $\mathrm{G}_2$ manifold is of type $\mathcal{X}_1 \oplus \mathcal{X}_3 \oplus \mathcal{X}_4$.
From Table~\ref{tableBesse}, for the function $f(t)= \sin \, t$, resp. $f(t)= \sinh \, t$,
the metric induced by the $\mathrm{G}_2$-structure is Einstein
with $Scal(g_{\varphi})=42$, resp. $Scal(g_{\varphi})=-42$.
For $f(t)=t$ the resulting metric is Ricci flat.

In the case (iii), we take $\alpha=3/c$. Since $|c|>3$ one has that $|\alpha|<1$ and we can take $\beta$ such that $\beta^2=1-\alpha^2=\frac{c^2-9}{c^2}$.
The torsion forms are
\begin{equation*}
\begin{aligned}
\tau_0 = &  \, \frac{8}{7f} \sqrt{c^2-9},\quad\ \
\tau_1 = \, \frac{1}{f}(f' - 1) dt, \quad\ \
\tau_2 = \,  - \frac{3}{c} f \, \sigma_2, \\[4pt]
\tau_3 = & \, - \frac{3\sqrt{c^2-9}}{7c} f^2\,\psi_+
+ \frac{c^2-9}{7c} f^2\,\psi_-
+ \frac{4}{21} \sqrt{c^2-9} f\,\omega \wedge dt
- \frac{\sqrt{c^2-9}}{c} f\, \sigma_2 \wedge dt, \\
\end{aligned}
\end{equation*}
The only possibility for a torsion form to be zero is to consider the function $f(t)=t$ to get $\tau_1=0$
(the other torsion forms are clearly non-zero).
Therefore, we obtain a $\mathrm{G}_2$ manifold of type $\mathcal{X}_1 \oplus \mathcal{X}_2 \oplus \mathcal{X}_3$ with Ricci flat induced metric.
\end{proof}

In order to exemplify this construction we describe first an example of Einstein coupled $\mathrm{SU}(3)$-structure
arising from a twistor space.

\begin{example}\label{twistor}
{\rm
It is well-known that the set of positive, orthogonal almost complex structures on a four-dimensional oriented Riemannian manifold forms
a smooth manifold $\mathcal{Z}$. The 6-dimensional manifold $\mathcal{Z}$, which is known as the twistor space, admits a (non-integrable) almost complex structure~$J$~\cite{ES}.
If in addition the four-manifold is self-dual and Einstein, then $(\mathcal{Z},J)$ admits an Einstein coupled  $\mathrm{SU}(3)$-structure~\cite{Tomasiello}.

We follow the lines of~\cite{FR1} for the description of this coupled structure.
There is a local frame $\{e^1,\ldots,e^6\}$ for the 1-forms on $\mathcal{Z}$ such that the coupled $\mathrm{SU}(3)$-structure
$(\omega,\psi_+,\psi_-)$
expresses locally as
$$
\omega=\frac{8}{5}(e^{12}+e^{34}+e^{56}), \quad
\psi_+=\mathfrak{R}\mathfrak{e}\, \Psi, \quad
\psi_-=\mathfrak{I}\mathfrak{m}\, \Psi,
$$
where
$$
\Psi=\left(8/5\right)^{\frac{3}{2}} i\, (e^1+ie^2)\wedge(e^3+ie^4)\wedge(e^5+ie^6).
$$
The differential of the forms $\omega$ and $\psi_-$ are given by
$$
d\omega=-\frac{3}{2} \sigma_0\, \psi_+, \quad\quad
d\psi_-= \sigma_0\, \omega^2 - \sigma_2 \wedge \omega,
$$
with
$$
\sigma_0=\frac{\sqrt{10}}{6}(\sigma+2),
\quad\quad
\sigma_2=-\frac{8\sqrt{10}}{15}(\sigma-1)\, (e^{12}+e^{34}-2e^{56}),
$$
where $24\, \sigma$ is equal to the scalar curvature of the given four-manifold.
The metric induced by the $\mathrm{SU}(3)$-structure is Einstein precisely for the values
$\sigma=1$ (in this case the torsion form $\sigma_2$ vanishes and the structure is nearly-K\"ahler)
and $\sigma=2$.
For the latter coupled $\mathrm{SU}(3)$-structure the constant $c$ in~\eqref{coupled-eq} is $c=-\sqrt{10}$, and
$$
Ric(g_{\omega, \psi_+})= 5\, g_{\omega, \psi_+},
$$
so that we can apply Theorem~\ref{coupled-alpha-beta-constantes}.

In the cases (i) and (ii) we get $\mathrm{G}_2$-structures which are locally conformal closed
or integrable (i.e. of types $\mathcal{X}_2 \oplus \mathcal{X}_4$ or $\mathcal{X}_1 \oplus \mathcal{X}_3 \oplus \mathcal{X}_4$)
whose induced metrics are Ricci flat for $f(t)=t$, Einstein with $Scal(g_{\varphi})=42$ for $f(t)= \sin \, t$,
and Einstein with $Scal(g_{\varphi})=-42$ for $f(t)= \sinh \, t$.

In the case (iii) of Theorem~\ref{coupled-alpha-beta-constantes}, since $|c|=\sqrt{10}>3$,
we get that the $\mathrm{G}_2$-structure
$$
\varphi = t^2 \omega \wedge dt - \frac{t^3}{\sqrt{10}} (3\,\psi_+ - \psi_-),
$$
is of type $\mathcal{X}_1 \oplus \mathcal{X}_2 \oplus \mathcal{X}_3$ with Ricci flat induced metric.
}
\end{example}

\begin{remark}
{\rm
Bryant proved in \cite{Br} that
there are no closed $\mathrm{G}_2$-structures $\varphi$ with $Scal(g_{\varphi}) \geq 0$ unless they are parallel.
Indeed, by~\eqref{scal7} any such structure satisfies $Scal(g_{\varphi})=-\frac{1}{2} |\tau_2|^2$.
From Example~\ref{twistor} it follows that such a result cannot be extended to the locally conformal closed class,
since there are (non parallel) Einstein examples with positive scalar curvature, as well as Ricci flat examples.
Notice that the latter case is considered by Fino and Raffero in~\cite{FR1}.
}
\end{remark}

In the following result we extend the case (iii) in Theorem~\ref{coupled-alpha-beta-constantes} to
more general $\mathrm{G}_2$-structures for which the functions $\alpha$ and $\beta$ are not constant.
This produces new Einstein examples with positive, as well as negative, scalar curvature
when we apply the result to a twistor space over a self dual Einstein $4$-manifold.

\begin{theorem}\label{coupled-alpha-beta-no-constantes}
Let $(L, \omega, \psi_+, \psi_-)$ be a (non nearly-K\"ahler) Einstein coupled $\mathrm{SU}(3)$ manifold with $Scal(g_{\omega, \psi_+})=30$.
Then,
\begin{enumerate}
\item[(i)]
the $\mathrm{G}_2$-structure $\varphi$ on the manifold $M=(0,\pi) \times L$ given by
$$
\varphi = \sin^2t\, \omega \wedge dt + \frac{\sin^3t}{c} \Big(3\cos t\,\psi_+ - \sqrt{c^2-9\cos^2t\, }\,\psi_-\Big)
$$
is of type $\mathcal{X}_1 \oplus \mathcal{X}_2 \oplus \mathcal{X}_3$
and its induced metric is Einstein with $Scal(g_{\varphi})=42$;
\item[(ii)]
the $\mathrm{G}_2$-structure $\varphi$ on the manifold $M=\left( 0,\ln \frac{|c|+\sqrt{c^2-9}}{3} \right) \times L$ given by
$$
\varphi = \sinh^2t\, \omega \wedge dt + \frac{\sinh^3t}{c} \Big(3\cosh t\,\psi_+ - \sqrt{c^2-9\cosh^2t\, }\,\psi_-\Big)
$$
is of type $\mathcal{X}_1 \oplus \mathcal{X}_2 \oplus \mathcal{X}_3$
and its induced metric is Einstein with $Scal(g_{\varphi})=-42$.
\end{enumerate}
\end{theorem}

\begin{proof}
By Proposition~\ref{coupled} we get that $\tau_1=0$ if and only if $\alpha(t)=\frac{3}{c} f'(t)$.

First we consider $f(t)=\sin t$. Since $|c|>3$ by Theorem~\ref{coupled-alpha-beta-constantes},
the function $\alpha(t)=\frac{3}{c} \cos t$
satisfies $|\alpha(t)|<1$ for any $t \in \mathbb{R}$.

Let us consider now $f(t)=\sinh t$. Since $|c|>3$, the function $\alpha(t)=\frac{3}{c} \cosh t$
satisfies $|\alpha(t)|\leq 1$ only for the values of
$t \in \left[-\ln \frac{|c|+\sqrt{c^2-9}}{3},\ln \frac{|c|+\sqrt{c^2-9}}{3} \right]$.

Hence, in both cases, the result follows by taking $\beta(t)$ such that
$\beta^2(t)=1-\alpha^2(t)$.
\end{proof}

Let us consider the twistor space $\mathcal{Z}$ over a self-dual Einstein $4$-manifold with the Einstein coupled $\mathrm{SU}(3)$-structure
given in Example~\ref{twistor}. Hence, from Theorem~\ref{coupled-alpha-beta-constantes}~(iii) and
Theorem~\ref{coupled-alpha-beta-no-constantes},
we obtain $\mathrm{G}_2$
manifolds in the class $\mathcal{X}_1 \oplus \mathcal{X}_2 \oplus \mathcal{X}_3$ which are Ricci flat, or Einstein with
$Scal(g_{\varphi})=\pm 42$.

\smallskip

Einstein $\mathrm{G}_2$ manifolds in the class $\mathcal{X}_2 \oplus \mathcal{X}_3 \oplus \mathcal{X}_4$ are given in the following

\begin{theorem}\label{2+3+4}
Let $(L, \omega, \psi_+, \psi_-)$ be a (non nearly-K\"ahler) Einstein coupled $\mathrm{SU}(3)$ manifold with $Scal(g_{\omega,\psi_+})=30$.
Let $c$ denote the coupled constant, and
consider $\theta(t)$ as follows:
\begin{enumerate}
\item[(i)] if $\theta(t)= \arcsin \left(\frac{2t^{-2c}}{1+t^{-4c}} \right)$, then the $\mathrm{G}_2$-structure $\varphi$ on the manifold $M= (0, \infty) \times L$ given by
$$
\varphi= t^2 \omega \wedge dt + t^3 (\cos \theta(t)\, \psi_+ - \sin \theta(t)\, \psi_-)
$$
belongs to the class $\mathcal{X}_2 \oplus \mathcal{X}_3 \oplus \mathcal{X}_4$ and its induced metric is Ricci flat;
\item[(ii)] if $\theta(t)= \arcsin \left(\frac{2(\tan \frac{t}{2})^{-2c}}{1+(\tan \frac{t}{2})^{-4c}}\right)$,
then the $\mathrm{G}_2$-structure $\varphi$ on the manifold $M= (0, \pi) \times L$ given by
$$
\varphi= \sin^2 t \omega \wedge dt + \sin^3 t (\cos \theta(t)\, \psi_+ - \sin \theta(t)\, \psi_-)
$$
belongs to the class $\mathcal{X}_2 \oplus \mathcal{X}_3 \oplus \mathcal{X}_4$
and its induced metric is Einstein with $Scal(g_{\varphi})=42$;
\item[(iii)] if $\theta(t)= \arcsin \left(\frac{2(\tanh \frac{t}{2})^{-2c}}{1+(\tanh \frac{t}{2})^{-4c}}\right)$,
then the $\mathrm{G}_2$-structure $\varphi$ on the manifold $M= (0, \infty) \times L$ given by
$$
\varphi= \sinh^2 t \omega \wedge dt + \sinh^3 t (\cos \theta(t)\, \psi_+ - \sin \theta(t)\, \psi_-)
$$
belongs to the class $\mathcal{X}_2 \oplus \mathcal{X}_3 \oplus \mathcal{X}_4$
and its induced metric is Einstein with $Scal(g_{\varphi})=-42$.
\end{enumerate}
\end{theorem}

\begin{proof}
Taking $\alpha(t)= \cos \theta(t)$ and $\beta(t)=\sin \theta(t)$ in
Proposition~\ref{coupled} we get that
$\tau_0 = \frac{4}{7f}(2 c \sin \theta  + f \theta')$.
A direct calculation shows that for (i), (ii) and (iii) with $f(t)=t$, $\sin t$ and $\sinh t$, respectively, the torsion form $\tau_0$ vanishes,
so the $\mathrm{G}_2$-structure belongs to the class $\mathcal{X}_2 \oplus \mathcal{X}_3 \oplus \mathcal{X}_4$ and the induced metric is Einstein.
Note that $\tau_1$, $\tau_2$ and $\tau_3$ never vanish.
\end{proof}

\end{subsection}

\begin{subsection}{Warped products of Einstein solvmanifolds}\label{sec-solvmanifolds}

Up to now, we have constructed Einstein warped $\mathrm{G}_2$ manifolds by means of the warping functions
$f(t)= e^t$, $\sinh t$, $t$ or $\sin t$. In view of Table~\ref{tableBesse},
it remains to obtain examples with warping function $f(t)=\cosh t$.
Note that in order to obtain such examples, the fiber manifold is required to be Einstein with negative scalar curvature.
For this reason, and since Einstein solvmanifolds have negative scalar curvature, in this section
we consider the warped products of 6-dimensional solvmanifolds.

An Einstein solvmanifold $(S, g)$ can be described in terms of its Einstein metric solvable Lie algebra, namely $(\mathfrak{s}, \langle \cdot, \cdot \rangle_{\mathfrak{s}})$, where $\mathfrak{s}$ is the Lie algebra of the solvable Lie group $S$, and
$\langle \cdot, \cdot \rangle_{\mathfrak{s}}$ is the scalar product on~$\mathfrak{s}$.
In \cite{La} Lauret obtained a structure theorem for Einstein metric solvable Lie algebras.

\begin{theorem}{\cite{La}}
Any Einstein metric solvable Lie algebra $(\mathfrak{s}, \langle \cdot, \cdot \rangle_{\mathfrak{s}})$
has to be of standard type.
\end{theorem}

Let $(\mathfrak{n}, \langle \cdot, \cdot \rangle)$ be a metric nilpotent Lie algebra. A metric solvable extension
of $(\mathfrak{n}, \langle \cdot, \cdot \rangle)$ is a metric solvable Lie algebra $(\mathfrak{s}, \langle \cdot, \cdot \rangle_{\mathfrak{s}})$
such that $\mathfrak{s}$ has the orthogonal decomposition $\mathfrak{s}=\mathfrak{n} \oplus \mathfrak{a}$ with $[\mathfrak{s},\mathfrak{s}]=\mathfrak{n}$, $[\mathfrak{a},\mathfrak{a}] \subset \mathfrak{n}$ and
$\langle \cdot, \cdot \rangle_{\mathfrak{s}}|_{\mathfrak{n} \times \mathfrak{n}}= \langle \cdot, \cdot \rangle$.
The metric solvable Lie algebra $(\mathfrak{s}, \langle \cdot, \cdot \rangle_{\mathfrak{s}})$ is said to be \emph{standard}
or to have \emph{standard type} if $\mathfrak{a}$ is an Abelian subalgebra of $\mathfrak{s}$.
In this case, $\dim \mathfrak{a}$ is called the \emph{rank}.

Taking into account the structure theorem, in \cite[Section 3.2]{Ma} a classification of Einstein metric 6-dimensional solvable Lie algebras
is obtained. There, metric nilpotent Lie algebras up to dimension five are considered,
and their corresponding Einstein metric solvable extensions are described.

By considering these $6$-dimensional Einstein metric solvable Lie algebras, in the following example we give
an Einstein $\mathrm{G}_2$ manifold obtained as a warped product with warping function $f(t) = \cosh t$.

\begin{example}\label{Sol-6-4}
{\rm
Let $(S,g)$ be the solvmanifold corresponding to the metric solvable Lie
algebra $(\mathfrak{s}, \langle \cdot, \cdot \rangle)$ with
\begin{equation*}
\mathfrak{s}= \left(\frac{\sqrt{10}}{4}e^{16},\, \frac{\sqrt{10}}{4}e^{26},\, \frac{\sqrt{10}}{4}e^{36},
\, \frac{\sqrt{10}}{4}e^{46},\, \frac{\sqrt{10}}{2}e^{12}+\frac{\sqrt{10}}{2}e^{34}+\frac{\sqrt{10}}{2}e^{56},\, 0\right),
\end{equation*}
and $\langle e^i, e^j \rangle = \delta_{ij}$. Consider the $\mathrm{SU}(3)$-structure $(\omega,\psi_+,\psi_-)$ on $S$ given by
\begin{equation*}
\begin{aligned}
\omega & = e^{12}+e^{34}+e^{56}, \\
\psi_+ & = e^{135}-e^{146}-e^{236}-e^{245}, \\
\psi_- & = e^{136}+e^{145}+e^{235}-e^{246}. \\
\end{aligned}
\end{equation*}
It is clear that the induced metric is precisely the given $g$, i.e. $g=g_{\omega,\psi_+}$, and it can be checked that
\begin{equation*}
Ric (g_{\omega, \psi_+}) = -5\, g_{\omega, \psi_+}.
\end{equation*}
A direct calculation shows that
$$
d\omega = 0,\quad\quad
d\psi_+ = \pi_1 \wedge \psi_+,\quad\quad
d\psi_- = \pi_1 \wedge \psi_-,
$$
where $\pi_1 = - \sqrt{10} \, e^6$ is the unique non-zero torsion of the $\mathrm{SU}(3)$-structure.

Thus, the $\mathrm{SU}(3)$ manifold $(S, \omega, \psi_+, \psi_-)$ is of type $\mathcal{W}_5$ and its induced metric is Einstein with $Scal(g_{\omega,\psi_+})=-30$.
We conclude that the $\mathrm{G}_2$ manifold $(\mathbb{R} \times S, \varphi)$ with
\begin{equation*}
\varphi = \cosh^2 t \, \omega \wedge dt + \cosh^3 t  \, \psi_+,
\end{equation*}
is of type $\mathcal{X}_2 \oplus  \mathcal{X}_3 \oplus \mathcal{X}_4$ and its induced metric is Einstein with $Scal(g_{\varphi})=-42$.
Indeed, by
Corollary~\ref{corolary1} we have $\tau_0=0$, and $\tau_1,\tau_2,\tau_3 \not=0$,
because $\pi_1 \not= 0= \nu_1$.
}
\end{example}

\end{subsection}

\end{section}

\begin{section}{Classification of Einstein $\mathrm{G}_2$-structures}\label{sec-classification-G2}

\noindent In this section we apply the results and constructions of Einstein $\mathrm{G}_2$-structures on warped product
given in the previous sections.
Motivated by the classification problem studied by
Cabrera, Monar and Swann in \cite{CMS}, we realize most of the $\mathrm{G}_2$-classes
in the Einstein setting with scalar curvature of different signs.
Moreover, at the end of the section we produce several explicit families of Einstein $\mathrm{G}_2$-structures with identical Riemannian metric
but having different $\mathrm{G}_2$ type (see \cite{AChFrH,Br,Grigorian,spiro,Lin} for related results).

In Table~\ref{tabla-resumen} we show concrete Einstein examples, when they exist, in the different Fern\'andez-Gray classes
of $\mathrm{G}_2$ manifolds. Since the examples are warped products, in the first column we indicate the fibre.
By $\mathcal{NK}$ and $\mathcal{CY}$ we mean a nearly K\"ahler manifold and a Calabi Yau manifold, respectively.
The fiber $S^3 \times S^3$ is the Einstein $\mathrm{SU}(3)$ manifold described in Example~\ref{example1}.
By $\mathcal{Z}$ we mean the twistor space over a self-dual Einstein $4$-manifold with the Einstein coupled $\mathrm{SU}(3)$-structure given in Example~\ref{twistor}. Finally,
$S$ is the Einstein solvmanifold given in Example~\ref{Sol-6-4}.

The second, third and fourth columns give information about the class of the $\mathrm{SU}(3)$-structure
on the fiber, the Einstein constant $\mu$ of its induced metric, and the torsion forms which are nonzero, respectively.

In Table~\ref{tabla-resumen} we also indicate the functions $f(t)$ that give rise to the Einstein $\mathrm{G}_2$ manifolds.
The functions $\alpha(t)=\cos \theta(t)$ and $\beta(t)=\sin \theta(t)$ defining the appropriate warped $\mathrm{G}_2$-structure
in each case are carefully chosen so that the resulting structure provides a strict example in the $\mathrm{G}_2$-class.
Here we use the term ``strict'' to indicate that the $\mathrm{G}_2$-structure does not belong to any subclass of the given one.
Next we give details for each $\mathrm{G}_2$-class:

\medskip

\noindent $\bullet$ \textbf{The class $\mathcal{P}$}.
Examples are given by the $t$-cone of a nearly K\"ahler manifold (see Proposition~\ref{sec-P} and Corollary~\ref{NK-characterizations}).

\medskip

\noindent $\bullet$ \textbf{Strict examples in $\mathcal{X}_1$}.
Strict examples are given in Proposition~\ref{sec-X1} (see also Corollaries~\ref{cor} and~\ref{NK-characterizations})
as the sine-cone of a nearly K\"ahler manifold.

\medskip

\noindent $\bullet$ \textbf{The classes $\mathcal{X}_2$ and $\mathcal{X}_3$}.
From Proposition~\ref{2+3-G2} (see also Corollary~\ref{no-existes-2-ni-3}) one has that via the warped construction
it is not possible to obtain strict Einstein examples in these classes.

\medskip

\noindent $\bullet$ \textbf{Strict examples in $\mathcal{X}_4$}.
Examples are given in Propositions~\ref{sec-X4-negative} and \ref{sec-X4-zero-positive} as warped products of Calabi-Yau manifolds
or, more generally,
of Einstein $\mathrm{SU}(3)$ manifolds in the class $\mathcal{W}_1^+ \oplus \mathcal{W}_1^-$.
For instance, for a nearly K\"ahler manifold, taking $\alpha(t)=1$ and $\beta(t)=0$ we get
Einstein examples in $\mathcal{X}_4\backslash \mathcal{P}$ with constant $\lambda=-6$ for $f(t)=\sinh t$,
and constant $\lambda=6$ for $f(t)=\sin t$.
Also Ricci flat examples in $\mathcal{X}_4\backslash \mathcal{P}$ can be obtained with the construction described in Proposition~\ref{sec-X4-zero-positive}~(i).

\medskip

\noindent $\bullet$ \textbf{The class $\mathcal{X}_1 \oplus \mathcal{X}_2$}.
On a connected manifold, one has that
$\mathcal{X}_1 \cup \mathcal{X}_2=\mathcal{X}_1 \oplus \mathcal{X}_2$ (see \cite[Theorem 2.1]{CMS}),
so there do not exist strict $\mathrm{G}_2$-structures in this class.
Notice that from Proposition~\ref{2+3-G2} we conclude that for Einstein warped $\mathrm{G}_2$ manifolds
one has $\mathcal{X}_1 \cup \mathcal{X}_2 = \mathcal{X}_1$.

\medskip

\noindent $\bullet$ \textbf{Strict examples in $\mathcal{X}_1 \oplus \mathcal{X}_3$}.
The $\mathrm{G}_2$-structures given in Example~\ref{example1} starting from
$S^3 \times S^3$ provide Einstein coclosed examples.
Moreover, using Corollary~\ref{corolary1} one can see that
the torsion forms $\tau_0, \tau_3 \not= 0$, so they are strict.

\medskip

\noindent $\bullet$ \textbf{Strict examples in $\mathcal{X}_1 \oplus \mathcal{X}_4$}.
A $\mathrm{G}_2$-structure belongs to
$\mathcal{X}_1 \oplus \mathcal{X}_4 \backslash (\mathcal{X}_1 \cup \mathcal{X}_4)$
if and only if the torsion forms satisfy $\tau_2=\tau_3=0$ and $\tau_0, \tau_1 \not= 0$.
In order to construct strict examples in the class $\mathcal{X}_1 \oplus \mathcal{X}_4$,
we consider a nearly-K\"ahler manifold %%%$\mathcal{NK}$
$L$, with torsion $\sigma_0=-2$ and Einstein constant $\mu=5$.
Let us take $\alpha(t)=\cos \theta(t)$ and $\beta(t)=\sin \theta(t)$, with function $\theta(t)$ chosen as follows:
\begin{enumerate}
\item[{\rm (i)}] if $\theta(t)=2 \arctan (e^C t)$, with $C$ a constant, and $f(t)=t$, then the corresponding
warped $\mathrm{G}_2$-structure on the manifold $(0,\infty) \times L$ belongs to
$\mathcal{X}_1 \oplus \mathcal{X}_4 \backslash (\mathcal{X}_1 \cup \mathcal{X}_4)$
and its induced metric is Ricci flat;
\item[{\rm (ii)}] if $\theta(t)=2 \arctan (e^C \tanh\frac{t}{2})$, with $C$ a constant, and $f(t)=\sinh t$, then we
get a warped $\mathrm{G}_2$-structure on $(0,\infty) \times L$ sitting in
$\mathcal{X}_1 \oplus \mathcal{X}_4 \backslash (\mathcal{X}_1 \cup \mathcal{X}_4)$
whose induced metric is Einstein with $\lambda=-6$;
\item[{\rm (iii)}] if $\theta(t)=2 \arctan (e^C \tan\frac{t}{2})$, with $C \not= 0$ a constant, and $f(t)=\sin t$, then we
get a warped $\mathrm{G}_2$-structure on the manifold $(0,\pi) \times L$ that belongs to
$\mathcal{X}_1 \oplus \mathcal{X}_4 \backslash (\mathcal{X}_1 \cup \mathcal{X}_4)$ and whose induced metric is Einstein with $\lambda=6$.
\end{enumerate}
Notice that if in the case (iii) one considers $C=0$, then one recovers the sine-cone over a nearly-K\"ahler manifold, and
so the $\mathrm{G}_2$-structure belongs to $\mathcal{X}_1 \backslash \mathcal{P}$.

For characterization results of manifolds in the strict class $\mathcal{X}_1 \oplus \mathcal{X}_4$, see \cite{CI-MRL}.

\medskip

\noindent $\bullet$ \textbf{The class $\mathcal{X}_2 \oplus \mathcal{X}_3$}.
By Proposition~\ref{2+3-G2} we have that via the warped product construction
it is not possible to obtain strict Einstein examples in the class
$\mathcal{X}_2 \oplus \mathcal{X}_3$.

\medskip

\noindent $\bullet$ \textbf{Strict examples in $\mathcal{X}_2 \oplus \mathcal{X}_4$}.
We consider the warped $\mathrm{G}_2$-structures in the class $\mathcal{X}_2 \oplus \mathcal{X}_4$ given in Example~\ref{twistor} starting from
the twistor space $\mathcal{Z}$ over a self-dual Einstein $4$-manifold.
Using Corollary~\ref{corolary1} one can see that
the torsion forms $\tau_1, \tau_2 \not= 0$,
so they belong to $\mathcal{X}_2 \oplus \mathcal{X}_4 \backslash (\mathcal{X}_2 \cup \mathcal{X}_4)$.

\medskip

\noindent $\bullet$ \textbf{Strict examples in $\mathcal{X}_3 \oplus \mathcal{X}_4$}.
For strict examples in $\mathcal{X}_3 \oplus \mathcal{X}_4$,
we consider the product manifold $S^3 \times S^3$ endowed with the $\mathrm{SU}(3)$-structure given in Example~\ref{example1}.
Recall that the torsion reduces to $\sigma_0=-\sqrt{5}$ and $\nu_3\not=0$.
A $\mathrm{G}_2$-structure belongs to $\mathcal{X}_3 \oplus \mathcal{X}_4 \backslash (\mathcal{X}_3 \cup \mathcal{X}_4)$
if and only if $\tau_0=\tau_2=0$ and $\tau_1, \tau_3 \not= 0$.

Taking $(\alpha,\beta)=(1,0)$, we get that the warped $\mathrm{G}_2$-structure
$\varphi = f^2 \omega \wedge dt + f^3 \psi_+$
on the manifold $M=I_f \times S^3 \times S^3$ satisfies $\tau_0=\tau_2=0$
and its induced metric is Ricci flat for $f(t)=t$,
Einstein with positive scalar curvature for $f(t)= \sin \, t$,
and Einstein with negative scalar curvature for $f(t)= \sinh \, t$.

Clearly, $\nu_3\not=0$ implies $\tau_3 \not= 0$ by Corollary~\ref{corolary1}. Moreover,
$\tau_1=0$ if and only if $\sigma_0\,\alpha+2 f'(t)=-\sqrt{5}+2 f'(t)=0$.
Hence, it is clear that $\tau_1\not= 0$ for the functions $f(t)= \sinh \, t$, $t$ or $\sin \, t$.
In conclusion, one has Einstein examples in
$\mathcal{X}_3 \oplus \mathcal{X}_4 \backslash (\mathcal{X}_3 \cup \mathcal{X}_4)$
with Einstein constant $\lambda=-6$, $0$ or $6$.

\medskip

\noindent $\bullet$ \textbf{Strict examples in the classes $\mathcal{X}_1 \oplus \mathcal{X}_2 \oplus \mathcal{X}_3$
and $\mathcal{X}_1 \oplus \mathcal{X}_3 \oplus \mathcal{X}_4$}.
Several strict examples in these classes are constructed in Section~\ref{sec-coupled} on warped products of
Einstein coupled $\mathrm{SU}(3)$ manifolds (see Theorems~\ref{coupled-alpha-beta-constantes}~(ii)--(iii)
and~\ref{coupled-alpha-beta-no-constantes}, and also Example~\ref{deform-2} below).

\medskip

\noindent $\bullet$ \textbf{The class $\mathcal{X}_1 \oplus \mathcal{X}_2 \oplus \mathcal{X}_4$}.
This class is the only one where the existence of a strict Einstein warped $\mathrm{G}_2$ manifold remains open.
An example could be obtained as follows.
Let $L$ be an Einstein $\mathrm{SU}(3)$-structure in the class $\mathcal{W}_1^- \oplus \mathcal{W}_4 \oplus \mathcal{W}_5$,
with Einstein constant $\mu=5$, and such that the nonzero torsion forms satisfy $\sigma_0=-2$ and $\nu_1= \pi_1 \not=0$.
The sine-cone of $L$, i.e. $\alpha(t)=\cos t$ and $\beta(t)=f(t)=\sin t$,
would satisfy that $\tau_3=0$ and $\tau_0,\tau_1,\tau_2\not=0$.
However, we do not know of any such $L$:
\begin{question}\label{su3-question}
{\rm
Are there Einstein $\mathrm{SU}(3)$-structures of positive scalar curvature
whose nonzero torsion is given by $\sigma_0=-2$ and $\nu_1= \pi_1 \not=0$?
}
\end{question}

\medskip

\noindent $\bullet$ \textbf{Strict examples in $\mathcal{X}_2 \oplus \mathcal{X}_3 \oplus \mathcal{X}_4$}.
Einstein examples in this class are given in Theorem~\ref{2+3+4} as a warped product of the twistor space $\mathcal{Z}$, and
in Example~\ref{Sol-6-4} as a warped product of the Einstein solvmanifold~$S$.
Since their torsion satisfies that $\tau_0=0$ and $\tau_1,\tau_2,\tau_3\not=0$, such examples are strict, i.e. they
belong to
$\mathcal{X}_2 \oplus \mathcal{X}_3 \oplus \mathcal{X}_4 \backslash
((\mathcal{X}_2 \oplus \mathcal{X}_3) \cup (\mathcal{X}_2 \oplus \mathcal{X}_4) \cup(\mathcal{X}_3 \oplus \mathcal{X}_4))$.

\medskip

\noindent $\bullet$ \textbf{Strict examples in the class general $\mathcal{X}_1 \oplus \mathcal{X}_2 \oplus \mathcal{X}_3 \oplus \mathcal{X}_4$}.
Examples on warped products of the twistor space $\mathcal{Z}$ are given in Example~\ref{deform-2} below.

\medskip
\medskip

We summarize the previous results in the following theorem. By ``admissible'' we mean that formula~\eqref{scal7}
does not give an obstruction to the existence of an Einstein $\mathrm{G}_2$-structure with
the desired scalar curvature in the given $\mathrm{G}_2$-class.

\hfill
\eject

\begin{theorem}\label{resumenG2}
For Einstein warped $\mathrm{G}_2$-structures, we have:
\begin{enumerate}
\item[{\rm (i)}]
There are Ricci flat warped $\mathrm{G}_2$-structures of every admissible strict type, except possibly for $\mathcal{X}_1\oplus\mathcal{X}_2\oplus\mathcal{X}_4$.
\item[{\rm (ii)}]
There are Einstein warped $\mathrm{G}_2$-structures with positive scalar curvature of every admissible strict type, except possibly for $\mathcal{X}_1\oplus\mathcal{X}_2\oplus\mathcal{X}_4$.
\item[{\rm (iii)}]
There are Einstein warped $\mathrm{G}_2$-structures with negative scalar curvature of every admissible strict type, except for $\mathcal{X}_2$, $\mathcal{X}_3$, $\mathcal{X}_2\oplus\mathcal{X}_3$, and possibly for $\mathcal{X}_1\oplus\mathcal{X}_2\oplus\mathcal{X}_4$.
\end{enumerate}
\end{theorem}

Motivated by these results, we ask the following general questions:

\begin{question}\label{Q1}
{\rm
Are there Einstein $\mathrm{G}_2$ manifolds in the strict class $\mathcal{X}_1\oplus\mathcal{X}_2\oplus\mathcal{X}_4$
with Einstein constant $<0$, $=0$, or $>0$?
}
\end{question}

\begin{question}\label{Q2}
{\rm
Are there Einstein $\mathrm{G}_2$ manifolds with negative scalar curvature in the strict classes
$\mathcal{X}_2$, $\mathcal{X}_3$ or $\mathcal{X}_2\oplus\mathcal{X}_3$?
}
\end{question}

\begin{remark}\label{remark-Q1-Q2}
{\rm
In~\cite[Section 8.4]{CI},
cohomogeneity-one metrics are used to
construct (Ricci-flat) metrics with holonomy in $\mathrm{G}_2$ and in different admissible $\mathrm{G}_2$-classes.
Concerning the class $\mathcal{X}_1\oplus\mathcal{X}_2\oplus\mathcal{X}_4$, one can see that the
vanishing of the torsion form $\tau_3$ implies that the functions defining the metric must be equal,
which leads to $\tau_2=0$ and so the $\mathrm{G}_2$-structure lies in $\mathcal{X}_1\oplus\mathcal{X}_4$.
}
\end{remark}

\medskip
\smallskip

The results in Sections~\ref{sec-coclosed} and~\ref{sec-coupled}
allow to construct explicit families of $\mathrm{G}_2$-structures in different classes but
with the same underlying Einstein metric.

For a fixed Riemannian metric generated by some $\mathrm{G}_2$-structure, it is natural to ask
what are the different $\mathrm{G}_2$-structures
that induce the same metric.
Bryant gave in \cite{Br} an answer to this general question, and recently Lin has investigated in~\cite{Lin} the space of parallel
$\mathrm{G}_2$-structures inducing the same Riemannian metric on a compact $7$-manifold.
In the following examples we provide some families of $\mathrm{G}_2$-structures in distinct classes but
with identical Einstein metric. We will consider deformations of the form
$$
\tilde{\varphi}=\varphi+\chi, \quad \mbox{ where } \chi=f^3(t) \big( A\,\alpha(t) \psi_+ - B\,\beta(t) \psi_-  \big)
$$
for certain constants $A,B$.
General results on deformations of the form $\tilde{\varphi}=\varphi+\chi$, where $\chi$ is a 3-form, are obtained in~\cite{Grigorian}
(see also \cite{spiro}).

\begin{example}\label{deform-1}
$\mathrm{G}_2$-structures with identical Einstein metric on warped products of a nearly K\"ahler manifold.
{\rm
Let us consider $L$ a nearly K\"ahler manifold and let $f(t)=\sin t$.
Following the case (iii) above of strict examples in $\mathcal{X}_1 \oplus \mathcal{X}_4$, consider
$\theta_C(t)=2 \arctan (e^C \tan\frac{t}{2})$, where $C \in \mathbb{R}$ is a constant.
The $\mathrm{G}_2$-structures $\varphi_{C}$
on $M=(0, \pi) \times L$ given by
$$
\varphi_{C} = \sin^2 t\, \omega \wedge dt + \sin^3 t \big( \cos\theta_C(t)\, \psi_+ - \sin\theta_C(t)\, \psi_-  \big).
$$
satisfy that $g_{\varphi_{C}} = dt^2+\sin^2 t \, g_L$, i.e. the induced Einstein metric is identical for all the $\mathrm{G}_2$-structures
in the family. The $\mathrm{G}_2$-type of $\varphi_{C}$ varies as follows:

\smallskip

$\bullet$
 $\mathcal{X}_1$, if and only if $C=0$;

\smallskip

$\bullet$
 $\mathcal{X}_1 \oplus \mathcal{X}_4$, if and only if $C\not= 0$.

\smallskip

\noindent Therefore, we can deform the structure in $\mathcal{X}_1$ to one in the class $\mathcal{X}_1 \oplus \mathcal{X}_4$.
}
\end{example}

\begin{example}\label{deform-2}
$\mathrm{G}_2$-structures with identical Einstein metric on warped products of the twistor space~$\mathcal{Z}$.
{\rm
We define an explicit family of $\mathrm{G}_2$-structures in different classes but
with the same induced Ricci flat metric starting from $L$ in the conditions of Theorem~\ref{coupled-alpha-beta-constantes} (iii),
in particular for $L=\mathcal{Z}$.
Let us
denote $\alpha_0=\frac{3}{c}$ and $\beta_0=\frac{\sqrt{c^2-9}}{c}$, and consider $(a,b) \in \mathbb{R}^2$ the points in the
ellipse of equation $\alpha_0^2\,a^2+\beta_0^2\,b^2=1$. On $M=(0,\infty) \times L$
we take the family of $\mathrm{G}_2$-structures
$$
\varphi_{a,b} = t^2 \omega \wedge dt + t^3 \left(a\,\alpha_0\,\psi_+ - b\,\beta_0\,\psi_-\right).
$$
The induced Ricci flat metric is $g_{\varphi_{a,b}} = dt^2+t^2 g_L$, but the $\mathrm{G}_2$-structure belongs to the strict class

\smallskip

$\bullet$
 $\mathcal{X}_1 \oplus \mathcal{X}_2 \oplus \mathcal{X}_3$, if and only if $(a,b)=(1,1)$;

\smallskip

$\bullet$
 $\mathcal{X}_2 \oplus \mathcal{X}_4$, if and only if $(a,b)=(\alpha_0^{-1},0)$;

\smallskip

$\bullet$
 $\mathcal{X}_1 \oplus \mathcal{X}_3 \oplus \mathcal{X}_4$, if and only if $(a,b)=(0,\beta_0^{-1})$;

\smallskip

$\bullet$
 $\mathcal{X}_1 \oplus \mathcal{X}_2 \oplus \mathcal{X}_3 \oplus \mathcal{X}_4$ for any other values of $(a,b)$.

\medskip

Similar families can be constructed for the other Einstein metrics based on $f(t)=\sin t$ and $f(t)=\sinh t$.
Take $(a,b) \in \mathbb{R}^2$ satisfying $a^2+b^2=1$. On $M=(0,\infty) \times L$, we consider
the family of $\mathrm{G}_2$-structures
$$
\varphi_{a,b} = f^2 \omega \wedge dt + f^3 \left(a\,\psi_+ - b\,\psi_-\right).
$$
The induced Einstein metric is $g_{\varphi_{a,b}} = dt^2+f^2 g_L$, and by Theorem~\ref{coupled-alpha-beta-constantes}~(i)~(ii),
the $\mathrm{G}_2$-structure belongs to the strict class

\smallskip

$\bullet$
 $\mathcal{X}_2 \oplus \mathcal{X}_4$, if and only if $(a,b)=(1,0)$;

\smallskip

$\bullet$
 $\mathcal{X}_1 \oplus \mathcal{X}_3 \oplus \mathcal{X}_4$, if and only if $(a,b)=(0,1)$;

\smallskip

$\bullet$
 $\mathcal{X}_1 \oplus \mathcal{X}_2 \oplus \mathcal{X}_3 \oplus \mathcal{X}_4$ for any other values of $(a,b)$.

\medskip

\noindent The Einstein constant is positive, resp. negative, for $f(t)=\sin t$, resp. $f(t)=\sinh t$.

}
\end{example}

\end{section}

%%%%%%%%%%%%%%%%%%%%%%%%%%%%%%%%%%%%%%%%%%%%%%%%%%
%%%%%%%%%%%%%%%%%%%%%%%%%%%%%%%%%%%%%%%%%%%%%%%%%%
%%%%%%%%%%%%%%%%%%%%%%%%%%%%%%%%%%%%%%%%%%%%%%%%%%

\begin{section}{$\mathrm{Spin}(7)$-structures}\label{sec-Spin(7)}

\noindent
In this section we consider $\mathrm{Spin}(7)$ manifolds given as a warped product of a $\mathrm{G}_2$ manifold, and
we obtain an explicit description of the torsion forms of the warped $\mathrm{Spin}(7)$-structure in terms of the torsion
forms of the $\mathrm{G}_2$-structure.

A $\mathrm{Spin}(7)$-structure on an 8-dimensional manifold $N$ consists in a reduction of the structure group of its frame bundle to the Lie group $\mathrm{Spin}(7)$. Equivalently, such structure can be characterized by the existence of a global non-degenerate
4-form $\phi$ on $N$ which can be locally written as
\begin{equation}\label{spin(7)can}
\begin{aligned}
\phi= &  \ e^{1278}+e^{3478}+e^{5678}+e^{1358}-e^{1468}-e^{2368}-e^{2458}\\[2pt]
              & + e^{1234}+e^{1256}+e^{3456}+e^{1367}+e^{1457}+e^{2357}-e^{2467},
\end{aligned}
\end{equation}
where $\{e^1,\dots, e^8\}$ is a local basis of 1-forms on $N$.
The presence of a $\mathrm{Spin(7)}$-structure on a manifold defines a Riemannian metric $g_{\phi}$ which can be obtained as
\begin{equation*}
g_{\phi}(e_i, e_j) = \frac{1}{24} \sum_{k \,l \,m} \phi(e_i, e_k, e_l, e_m) \, \phi(e_j, e_k, e_l, e_m).
\end{equation*}

Given a $\mathrm{Spin}(7)$ manifold $(N,\phi)$, the group $\mathrm{Spin}(7)$ acts on the space of differential $p$-forms $\Omega^p(N)$ on~$N$.
This action is irreducible on $\Omega^1(N)$ and $\Omega^7(N)$, but it is reducible for $\Omega^p(N)$ with $2 \leq p \leq 6$.
Since the Hodge star operator $\ast_{8}$ induces an isomorphism
$\ast_{8}\, \Omega^p(N) \cong \Omega^{8-p}(N)$,
it suffices to describe the decompositions for $p=2, 3$ and $4$.
In \cite{Br0} it is shown that the $\mathrm{Spin(7)}$ irreducible decompositions for $2\leq p\leq 4$ are
\begin{equation*}
\begin{aligned}
\Omega^2(N) & = \Omega^2_7(N) \oplus \Omega^2_{21}(N),\\
\Omega^3(N) & = \Omega^3_8(N) \oplus \Omega^3_{48}(N),\\
\Omega^4(N) & = \Omega^4_1(N) \oplus \Omega^4_{7}(N) \oplus \Omega^4_{27}(N) \oplus \Omega^4_{35}(N),
\end{aligned}
\end{equation*}
where $\Omega^p_k(N)$ denotes the $\mathrm{Spin}(7)$ irreducible space of $p$-forms of dimension $k$ at every point.
The description on the other degrees is obtained via the isomorphism $\ast_{8}\, \Omega^p_k(N) \cong \Omega^{8-p}_k(N)$
given by the Hodge star operator, and in this section we are only interested in the $\mathrm{Spin}(7)$-type decomposition of 5-forms.
This space decomposes as
\begin{equation*}
\Omega^5(N) = \Omega^5_8(N) \oplus \Omega^5_{48}(N),
\end{equation*}
where
\begin{equation*}
\begin{aligned}
\Omega^5_8(N) & = \{\alpha \wedge \phi \mid \ \alpha \in \Omega^1(N)  \}, \\[2pt]
\Omega^5_{48}(N) & = \{ \gamma \in \Omega^5(N) \mid \ \phi \wedge \ast_{8} \gamma = 0 \}.
\end{aligned}
\end{equation*}

The isomorphisms between $\mathrm{Spin}(7)$ irreducible spaces introduce a scaling factor
on 1-forms $\kappa \in \Omega^1(N)$ as follows:
\begin{equation}\label{iso}
\ast_8\left(\ast_8(\kappa \wedge \phi) \wedge \phi \right) = - 7\kappa.
\end{equation}

%%%%%%%%%%%%%%%%%%%%%%%%%%%%%%%%%%%%%%%%%%%%%%%%%%%%%%
%%%%%%%%%%%%%%%%%%%%%%%%%%%%%%%%%%%%%%%%%%%%%%%%%%%%%%

\smallskip

The above decomposition of 5-forms on $N$ allows to express the exterior derivative of $\phi$ as
\begin{equation}\label{torsionforms2}
d\phi= \lambda_1 \wedge \phi + \lambda_5,
\end{equation}
where $\lambda_1 \in \Omega^1(N)$ and $\lambda_5 \in \Omega^5_{48}(N)$ are called the \emph{torsion forms} of the $\mathrm{Spin}(7)$-structure.

According to \cite{Fe} the covariant derivative of $\phi$ can be decomposed into two components, namely $Y_1$ and $Y_2$. Thus, a $\mathrm{Spin}(7)$-structure is said of type $\mathcal{P}, \mathcal{Y}_1, \mathcal{Y}_2$ or $\mathcal{Y}=\mathcal{Y}_1\oplus\mathcal{Y}_2$
if the covariant derivative $\nabla^{g_\phi}\phi$ lies in $\{0\}, Y_1, Y_2$ or $Y=Y_1 \oplus Y_2$, respectively.
In terms of the torsion forms, these classes are characterized
in Table~\ref{classes-Spin7}.
In the parallel case, the holonomy reduces to $\mathrm{Spin}(7)$ and the metric is Ricci-flat. Examples of
manifolds with $\mathrm{Spin}(7)$ holonomy are constructed in \cite{Br,BS,Joyce2}.

\medskip

\begin{table}[h]
\caption{\textbf{Classes of $\mathrm{Spin}(7)$-structures}}
\begin{center}
  \begin{tabular}{ |c | c | l |}
    \hline
    Class & Torsion forms & Structure\\ \hline \hline
    $\mathcal{P}$ & $\lambda_1= \lambda_5=0$ & Parallel \\ \hline
    $\mathcal{Y}_1$ & $\lambda_5=0$ & Locally conformal parallel \\ \hline
    $\mathcal{Y}_2$ & $\lambda_1=0$ & Balanced \\ \hline
    $\mathcal{Y}=\mathcal{Y}_1\oplus\mathcal{Y}_2$ & No condition & General $\mathrm{Spin}(7)$ \\
    \hline
  \end{tabular}
\end{center}
\label{classes-Spin7}
\end{table}

As it happened for $\mathrm{SU}(3)$ and $\mathrm{G}_2$ manifolds, the scalar curvature of
a $\mathrm{Spin}(7)$ manifold can be described in terms of the torsion forms.
The expression can be achieved from the formulas described in \cite{Iv,Pu} and is given as follows:
\begin{equation}\label{scal8}
Scal(g_{\phi})=\frac{21}{8}|\lambda_1|^2-\frac{1}{2}|\lambda_5|^2+\frac{7}{2}\, d^{*_8} \lambda_1,
\end{equation}
where $d^{*_8}$ denotes the codifferential, i.e. the adjoint operator of the exterior derivative with respect to the metric.

\medskip

Consider a 7-dimensional manifold $M$ endowed with a $\mathrm{G}_2$-structure $\varphi$.
Let $N$ be the Riemannian product $N=\mathbb{R} \times M$, and denote by
$p \colon N \longrightarrow \mathbb{R}$ and $q \colon N \longrightarrow M$
the projections. Then, the $4$-form
\begin{equation*}
\phi = q^*(\varphi) \wedge p^*(dt) + q^*(\ast_{7}\varphi),
\end{equation*}
with $t$ the coordinate on $\mathbb{R}$, defines a $\mathrm{Spin}(7)$-structure on $N$.
In the following, $\varphi$ and $\ast_{7}\varphi$ will be identified with their pullbacks onto $N$.
More generally, we have

\begin{proposition}
Let $(M, \varphi)$ be a $\mathrm{G}_2$ manifold and consider a function $f \colon I_f \longrightarrow \mathbb{R}$.
Then, the $4$-form on $N=I_f \times M$ given by
\begin{equation}\label{4-form}
\phi=f^3(t)\, \varphi \wedge dt + f^4(t) \ast_{7}\!\varphi
\end{equation}
defines a $\mathrm{Spin}(7)$-structure whose induced metric is
\begin{equation*}
g_{\phi}=f^2(t)\, g_{\varphi} + dt^2.
\end{equation*}
\end{proposition}

\begin{proof}
Let $\{e^1,\dots,e^7\}$ be a local orthonormal basis of 1-forms such that the 3-form $\varphi$ writes as in \eqref{G2can}.
Now, with respect to the local basis on $N$ given by
$\{h^1,\dots,h^8\}=\{f(t)e^1,\dots,f(t)e^7,dt\}$,
the 4-form $\phi$ can be written as in \eqref{spin(7)can}.
Therefore, $\{h^1,\dots, h^8\}$ is orthonormal for the metric $g_{\phi}$,
and
\begin{equation*}
g_{\phi}=\sum_{i=1}^8 h^i \otimes h^i=f^2(t) \sum_{i=1}^7 e^i \otimes e^i+ dt \otimes  dt = f^2(t)\, g_{\varphi} + dt^2.
\end{equation*}
\end{proof}

By the preceding proposition,
the $\mathrm{Spin}(7)$ manifold $N=I_f \times M$
with $\phi$ described in~\eqref{4-form} corresponds, as a Riemannian manifold, to the warped product $N=I_f \times_{f} M$.
We will refer to such a $\mathrm{Spin}(7)$-structure as a \emph{warped $\mathrm{Spin}(7)$-structure},
and the manifold $(N=I_f \times M,\phi)$ will be called \emph{warped $\mathrm{Spin}(7)$ manifold}.

\begin{lemma}\label{lemma2}
Let $\beta \in \Omega^q(M)$ be a differential $q$-form on $M$,
and let $\ast_7$ and $\ast_8$ be the Hodge star operators induced by the structures $\varphi$ and $\phi$, respectively. Then,
$$
\ast_8\beta = f^{7-2q} \ast_7\! \beta \wedge dt, \quad\quad
\ast_8(\beta \wedge dt) = (-1)^{q+1} f^{7-2q}\ast_7\! \beta.
$$
\end{lemma}

\begin{proof}
It is a consequence of the fact that the Hodge star operator $\ast_8$ is determined by
$(g_{\phi}, vol_8)$, where $vol_8=f^7 vol_7\wedge dt$ and $vol_7=\frac{1}{7}\varphi\wedge\ast_7\varphi$.
\end{proof}

\begin{theorem}\label{torsionesSpin7}
Let $(M, \varphi)$ be a $\mathrm{G}_2$ manifold with torsion forms $\tau_0, \tau_1, \tau_2, \tau_3$.
Then, the torsion forms $\lambda_1, \lambda_5$ of a warped $\mathrm{Spin}(7)$ manifold $(N=I_f \times M, \phi)$ are given by
\begin{equation*}%\label{torsion-spin7}
\begin{aligned}
\lambda_1   = &  \, \frac{1}{f}(\tau_0+4f')\, dt + \frac{24}{7} \tau_1,\\
\lambda_5   = &  \, -\frac{3}{7} f^3\, \tau_1\wedge \varphi \wedge dt
+ \frac{4}{7} f^4\, \tau_1\wedge \ast_7 \varphi + f^4\, \tau_2\wedge \varphi + f^3 \ast_7\!\tau_3 \wedge dt.
\end{aligned}
\end{equation*}
\end{theorem}

\begin{proof}
From \eqref{iso} and \eqref{torsionforms2}, and since $\lambda_5 \in \Omega^5_{48}(N)$, it follows that
the torsion form $\lambda_1$ is given by
$$
\lambda_1 = -\frac{1}{7} \ast_8\big((\ast_8 d\, \phi) \wedge \phi\big).
$$
In order to compute $\ast_8 d \phi$, we first take into account \eqref{torsionforms} and \eqref{4-form} to get
$$
d \phi  = f^3 (\tau_0 +4 f')\, \ast_7\! \varphi\wedge dt +3 f^3 \, \tau_1 \wedge \varphi \wedge dt + f^3 \ast_7\!\tau_3 \wedge dt
+ 4 f^4 \, \tau_1 \wedge \ast_7 \varphi + f^4 \, \tau_2 \wedge \varphi.
$$
A direct calculation using Lemma~\ref{lemma2} shows that
$$
\ast_8 d \phi  = -f^2 (\tau_0 +4 f')\, \varphi -3 f^2 \ast_7\! (\tau_1 \wedge \varphi) -f^3 \tau_3
+ 4 f \ast_7\! (\tau_1 \wedge \ast_7 \varphi) \wedge dt + f \ast_7\! (\tau_2 \wedge \varphi) \wedge dt.
$$
Now, by \eqref{G2prop} we arrive at
\begin{equation*}
\begin{aligned}
(\ast_8 d \phi) \wedge \phi = & \, -f^6 (\tau_0 +4 f')\, \varphi  \wedge \ast_7 \varphi
-3 f^5 \ast_7\! (\tau_1 \wedge \varphi) \wedge \varphi \wedge dt - 3 f^6 \ast_7\! (\tau_1 \wedge \varphi) \wedge \ast_7 \varphi \\
&  \, + 4 f^5 \ast_7\! (\tau_1 \wedge \ast_7 \varphi) \wedge \ast_7 \varphi \wedge dt
-f^6 (\tau_0 +4 f')\, \varphi  \wedge \ast_7 \varphi + 24\, f^5 \ast_7\!\tau_1 \wedge dt.
\end{aligned}
\end{equation*}
Then, using again Lemma~\ref{lemma2}, we get
\begin{equation*}
\ast_8 \big((\ast_8 d \phi) \wedge \phi  \big) = -\frac{7}{f}(\tau_0 +4 f')\, dt - 24\, \tau_1,
\end{equation*}
concluding that
\begin{equation*}
\lambda_1 = \, \frac{1}{f}(\tau_0+4f')\, dt + \frac{24}{7} \tau_1.
\end{equation*}

Finally, for the torsion form $\lambda_5$ we use that $\lambda_5= d \phi - \lambda_1 \wedge \phi$, together with the expressions of
$d\phi$ and $\lambda_1$ given above.
\end{proof}

A direct consequence of the previous theorem is the following

\begin{corollary}\label{corolary2}
The torsion forms of a warped $\mathrm{Spin}(7)$-structure satisfy:
\begin{equation*}
\begin{aligned}
\lambda_1  = 0  & \iff    \left \{  \begin{array}{ll} i)   \, \, & \,  \tau_0 + 4 f'=0,\\[3pt]
                                       ii)  \,  \, & \, \tau_1=0.
\end{array}
\right.
\\
\lambda_5  = 0 & \iff \left \{  \begin{array}{ll} iii)   &\tau_1=0, \\[3pt]
                                       iv)  &\tau_2=0, \\[3pt]
                                       v)  &\tau_3=0.
\end{array}
\right.
 \\
\end{aligned}
\end{equation*}
\end{corollary}

\end{section}

%%%%%%%%%%%%%%%%%%%%%%%%%%%%%%%%%%%%%%%%%%%%%%%%%%%%%%
%%%%%%%%%%%%%%%%%%%%%%%%%%%%%%%%%%%%%%%%%%%%%%%%%%%%%%

\begin{section}{Einstein warped $\mathrm{Spin}(7)$ manifolds}\label{sec-Einstein-Spin(7)}

\noindent
Our aim in this section is to construct Einstein $8$-manifolds in the different $\mathrm{Spin}(7)$-classes by means
of warped products of certain Einstein $\mathrm{G}_2$ manifolds, i.e. by means of warped $\mathrm{Spin}(7)$-structures.
As in Section~\ref{sec-Einstein-G2}, in order to use directly Table~\ref{tableBesse},
in this section we will also consider the Einstein metrics to be ``normalized''.

We begin with a characterization of the warped $\mathrm{Spin}(7)$ manifolds that are parallel,
which is related to a well known result in \cite{Bar}.

\begin{proposition}\label{Spin(7)P}
There exists a parallel warped $\mathrm{Spin}(7)$-structure on $N=I_f \times M$ if and only if the fiber $(M, \varphi)$ belongs to $\mathcal{X}_1$,
i.e. it is a nearly parallel $\mathrm{G}_2$ manifold, with torsion $\tau_0=-4$.

Furthermore, in that case $N=(0, \infty) \times M$ is the cone with $\mathrm{Spin}(7)$-structure
$$
\phi= t^3 \,  \varphi \wedge dt + t^4 \, \ast_7 \varphi.
$$
\end{proposition}

\begin{proof}
The parallel condition on the $\mathrm{Spin}(7)$-structure is equivalent to
$\lambda_1= \lambda_5 =0$.
From Corollary \ref{corolary2}, and taking into account the possible functions in Table~\ref{tableBesse},
these equations are equivalent to
$$\tau_1 = \tau_2 = \tau_3 = 0, \qquad \tau_0=-4 \qquad \text{ and } \qquad f(t)=t,$$
and the result follows.
\end{proof}

\medskip

The following three propositions give characterizations of the warped $\mathrm{Spin}(7)$ manifolds that are
Einstein and locally conformal parallel, depending on the sign of its scalar curvature.

\begin{proposition}\label{Spin(7)Z11}
There exists an Einstein locally conformal parallel warped $\mathrm{Spin}(7)$-structure $\phi$ on $N=I_f \times M$ with $Scal(g_{\phi})=56$
if and only if the fiber $(M, \varphi)$ belongs to $\mathcal{X}_1$ with
torsion $\tau_0=\pm 4$.

Furthermore, in that case $N=(0, \pi) \times M$ is the sine-cone with $\mathrm{Spin}(7)$-structure
$$
\phi= \sin^3 t \, \varphi \wedge dt + \sin^4 t \, \ast_7 \varphi.
$$
\end{proposition}

\begin{proof}
Suppose there exists such a warped product $(N=I_f \times M,\phi)$. Since $\lambda_5=0$, Corollary~\ref{corolary2}
forces the $\mathrm{G}_2$-structure $\varphi$ to be in the class $\mathcal{X}_1$.
Since $Scal(g_{\phi})=56$,
by Table~\ref{tableBesse} we get that the warping function is necessarily given by $f(t)= \sin t$
and $Scal(g_{\varphi})=42$.
Now, by~\eqref{scal7}, the torsion of the $\mathrm{G}_2$-structure is $\tau_0=\pm 4$.

Conversely, if we consider a nearly parallel $\mathrm{G}_2$ manifold with torsion $\tau_0=\pm 4$,
then the warped $\mathrm{Spin}(7)$-structure with $f(t)=\sin t$ is Einstein (with constant $7$)
and locally conformal parallel by Corollary~\ref{corolary2}.
\end{proof}

\begin{proposition}\label{Spin(7)Z10}
There exists a Ricci flat (strict) locally conformal parallel warped $\mathrm{Spin}(7)$-structure $\phi$ on $N=I_f \times M$
if and only if the fiber $(M, \varphi)$ belongs to $\mathcal{X}_1$ with torsion
$\tau_0=4$.

Furthermore, in that case $N=(0, \infty) \times M$ is the cone with $\mathrm{Spin}(7)$-structure
$$
\phi= t^3 \, \varphi \wedge dt + t^4 \ast_7 \varphi.
$$
\end{proposition}

\begin{proof}
The proof is similar to that of Proposition~\ref{Spin(7)Z11}, but taking into account that the Ricci flatness forces
the warping function to be $f(t)=t$. Hence, the locally conformal parallel $\mathrm{Spin}(7)$-structure is strict
only when $\tau_0=4$.
\end{proof}

\begin{proposition}\label{Spin(7)Z12}
There exists an Einstein locally conformal parallel warped $\mathrm{Spin}(7)$-structure $\phi$ on $N=I_f \times M$ with $Scal(g_{\phi})=-56$
if and only if the $\mathrm{G}_2$-structure $\varphi$ on the fiber $M$ is one of the following:

\smallskip

\noindent
$\bullet$ Parallel, and then $N = \mathbb{R} \times M$ is the exponential-cone with the $\mathrm{Spin}(7)$-structure
$$\phi= e^{3t}\, \varphi \wedge dt + e^{4t} \ast_7\! \varphi;$$

\noindent
$\bullet$ Nearly parallel with torsion $\tau_0=\pm 4$, and then $N = (0,\infty) \times M$ is the hyperbolic sine-cone with the

\hskip-.04cm $\mathrm{Spin}(7)$-structure $$\phi= \sinh^{3} t \, \varphi \wedge dt +\sinh^{4}  t \, \ast_7\! \varphi.$$
\end{proposition}

\begin{proof}
The proof is similar to the preceding propositions, but since $Scal(g_{\phi})=-56$,
by Table~\ref{tableBesse} we have that either $\tau_0=0$ and $f(t)= e^t$,
or $Scal(g_{\varphi})=42$ and $f(t)= \sinh t$.
In the first case the fiber is parallel, and in the second case it is a nearly parallel $\mathrm{G}_2$ manifold with torsion $\tau_0=\pm 4$.
\end{proof}

As a consequence one gets Einstein locally conformal parallel $\mathrm{Spin}(7)$ manifolds with
negative, zero or positive constant (see Corollary~\ref{cor} for $\mathrm{G}_2$ manifolds
satisfying the hypothesis of the following corollary).

\begin{corollary}\label{nuevo-fin}
Let $(M, \varphi)$ be a nearly parallel $\mathrm{G}_2$ manifold with torsion $\tau_0=4$. Then,
there are warped $\mathrm{Spin}(7)$-structures with fiber $(M, \varphi)$ which are (strict) locally conformal parallel and Einstein
with constant~$-7$, $0$ or $7$, by taking the function $f(t)=\sinh t$, $t$ or $\sin t$, respectively.
\end{corollary}

In the following result we note that there are no Einstein (strict) balanced warped $\mathrm{Spin}(7)$ manifolds.

\begin{proposition}\label{Spin(7)Z2}
A warped $\mathrm{Spin}(7)$ manifold is balanced and Einstein if and only if it is a parallel $\mathrm{Spin}(7)$ manifold.
\end{proposition}

\begin{proof}
Given an Einstein balanced warped $\mathrm{Spin}(7)$ manifold, since $\lambda_1=0$, from Corollary~\ref{corolary2}
we get that the torsion forms of the $\mathrm{G}_2$-structure on the fiber satisfy
$$\tau_1=0, \qquad \tau_0 = -4 $$
and the warping function in Table~\ref{tableBesse} is $f(t)=t$.
Thus, the $\mathrm{Spin}(7)$-structure is necessarily Ricci flat and by \eqref{scal8} we get $\lambda_5=0$.
In conclusion, the warped $\mathrm{Spin}(7)$-structure is parallel.
\end{proof}

\medskip

\medskip

As in Section~\ref{sec-classification-G2}, we summarize in Table~\ref{tabla-resumen2} the results obtained above for Einstein
warped $\mathrm{Spin}(7)$ manifold in the different strict classes:

\medskip

\noindent $\bullet$ \textbf{The class $\mathcal{P}$}.
Examples are given by the $t$-cone of a nearly parallel $\mathrm{G}_2$ manifold (see Proposition~\ref{Spin(7)P}).

\medskip

\noindent $\bullet$ \textbf{The class $\mathcal{Y}_1$}.
Strict examples with Einstein constant~$-7$, $0$ or $7$ are given in Corollary~\ref{nuevo-fin}
as the hyperbolic sine-cone, cone or sine-cone, respectively, of a nearly parallel $\mathrm{G}_2$ manifold with torsion $\tau_0=4$.

\medskip

\noindent $\bullet$ \textbf{The class $\mathcal{Y}_2$}.
By Proposition~\ref{Spin(7)Z2} it is not possible to obtain strict Einstein examples
via the warped construction.

\medskip

\noindent $\bullet$ \textbf{The general class $\mathcal{Y}_1 \oplus \mathcal{Y}_2$}.
Strict examples with positive, null and negative scalar curvature can be achieved as the different cones
of Einstein locally conformal parallel $\mathrm{G}_2$ manifolds
(see Section~\ref{sec-classification-G2} for examples of such $\mathrm{G}_2$ manifolds).

\medskip

We summarize the previous results in the following

\begin{theorem}\label{resumenSpin7}
For Einstein warped $\mathrm{Spin}(7)$-structures, we have:
\begin{enumerate}
\item[{\rm (i)}]
There are Ricci flat warped $\mathrm{Spin}(7)$-structures of every admissible strict type.
\item[{\rm (ii)}]
There are Einstein warped $\mathrm{Spin}(7)$-structures with positive scalar curvature of every admissible strict type.
\item[{\rm (iii)}]
There are Einstein warped $\mathrm{Spin}(7)$-structures with negative scalar curvature of every admissible strict type, except for $\mathcal{Y}_2$.
\end{enumerate}
\end{theorem}

Motivated by this result, we ask the following question:

\begin{question}\label{Q-spin7}
{\rm
Are there Einstein (non parallel) balanced $\mathrm{Spin}(7)$ manifolds?
}
\end{question}

%%%%%%%%%%%%%%%%%%%%%%%%%%%%%% tabla resumen
\newpage

\begin{landscape}

\begin{table}[h]
\caption{\textbf{Einstein warped $\mathrm{G}_2$-structures}}
\begin{tabular}{|c|c|c|c|c|c|c|c|}\hline
Fiber & $\mathrm{SU}(3)$-class & $\mu$ & $\mathrm{SU}(3)$ non-vanishing &  $f(t)$-cone metric & Strict $\mathrm{G}_2$-class & Einstein constant & $\mathrm{G}_2$ non-vanishing \\
&           &           & torsion forms             &       &                  &          $\lambda$        & torsion forms \\ \hline \hline
$\mathcal{NK}$ & $\mathcal{W}_1^-$ & $5$ & $\sigma_0 =\!-2$
& $t$ & $\mathcal{P}$ & $0$ & -- \\
\hline
$\mathcal{NK}$ & $\mathcal{W}_1^-$ & $5$ & $\sigma_0 =\!-2$
& $\sin t$ & $\mathcal{X}_1$ & $6$ & $\tau_0$ \\
\hline
$\nexists$ &  & & &  & $\mathcal{X}_2$ & & \\
\hline
$\nexists$ &  & & &  & $\mathcal{X}_3$ & & \\
\hline
$\mathcal{CY}$ & $\{0\} $ & $0$ & -- & $e^t$ & $\mathcal{X}_4$ & $-6$ & $\tau_1$ \\
\cline{1-5} \cline{7-7}
$\mathcal{NK}$ & $\mathcal{W}_1^-$ & $5$ & $\sigma_0 =\!-2$  & $\sinh t,\ t,\ \sin t$ &  & $-6,\ 0,\ 6$ &  \\
\hline
$\nexists$ & & & &  & $\mathcal{X}_1 \oplus \mathcal{X}_2= \mathcal{X}_1 \cup \mathcal{X}_2$ & & \\
\hline
$S^3 \times S^3$ & $\mathcal{W}_1^- \oplus \mathcal{W}_3 $ & $5$ & $\sigma_0,\ \nu_3$
& $\sinh t,\ t,\ \sin t$ & $\mathcal{X}_1 \oplus \mathcal{X}_3$ & $-6,\ 0,\ 6$ & $\tau_0,\ \tau_3$ \\
\hline
$\mathcal{NK}$ & $\mathcal{W}_1^-$ & $5$ & $\sigma_0 =\!-2$
& $\sinh t,\ t,\ \sin t$ & $\mathcal{X}_1 \oplus \mathcal{X}_4$ & $-6,\ 0,\ 6$ & $\tau_0,\ \tau_1$ \\
\hline
$\nexists$ &  & &  & & $\mathcal{X}_2 \oplus \mathcal{X}_3$ & & \\
\hline
$\mathcal{Z}$ & $\mathcal{W}_1^- \oplus \mathcal{W}_2^- $  & $5$ & $\sigma_0,\ \sigma_2$
& $\sinh t,\ t,\ \sin t$ & $\mathcal{X}_2 \oplus \mathcal{X}_4$ & $-6,\ 0,\ 6$ & $\tau_1,\ \tau_2$ \\
\hline
$S^3 \times S^3$ & $\mathcal{W}_1^- \oplus \mathcal{W}_3 $ & $5$ & $\sigma_0,\ \nu_3$ &
$\sinh t,\ t,\ \sin t$ & $\mathcal{X}_3 \oplus \mathcal{X}_4$ & $-6,\ 0,\ 6$ & $\tau_1,\ \tau_3$ \\
\hline

$\mathcal{Z}$ & $\mathcal{W}_1^- \oplus \mathcal{W}_2^- $  & $5$ & $\sigma_0,\ \sigma_2$
& $\sinh t,\ t,\ \sin t$ & $\mathcal{X}_1 \oplus  \mathcal{X}_2 \oplus \mathcal{X}_3$ & $-6,\ 0,\ 6$ & $\tau_0,\ \tau_2,\ \tau_3$ \\
\hline
(?) &   &  &
&  &  $\mathcal{X}_1 \oplus \mathcal{X}_2 \oplus \mathcal{X}_4$ &  & $\tau_0,\ \tau_1,\ \tau_2$ \\
\hline

$\mathcal{Z}$ & $\mathcal{W}_1^- \oplus \mathcal{W}_2^- $  & $5$ & $\sigma_0,\ \sigma_2$
& $\sinh t,\ t,\ \sin t$ & $\mathcal{X}_1 \oplus \mathcal{X}_3 \oplus \mathcal{X}_4$ & $-6,\ 0,\ 6$ & $\tau_0,\ \tau_1,\ \tau_3$ \\
\hline

$S$ & $\mathcal{W}_5 $  & $-5$ & $\pi_1$
& $\cosh t$ & $\mathcal{X}_2 \oplus \mathcal{X}_3 \oplus \mathcal{X}_4$ & $-6$ & $\tau_1,\ \tau_2,\ \tau_3$  \\

\cline{1-5} \cline{7-7}

$\mathcal{Z}$ & $\mathcal{W}_1^- \oplus \mathcal{W}_2^- $  & $5$ & $\sigma_0,\ \sigma_2$
& $\sinh t,\ t,\ \sin t$ &  & $-6,\ 0,\ 6$ &  \\
\hline

$\mathcal{Z}$ & $\mathcal{W}_1^- \oplus \mathcal{W}_2^- $  & $5$ & $\sigma_0,\ \sigma_2$
& $\sinh t,\ t,\ \sin t$ & $\mathcal{X}_1 \oplus \mathcal{X}_2 \oplus \mathcal{X}_3 \oplus \mathcal{X}_4$
& $-6,\ 0,\ 6$ & $\tau_0,\ \tau_1,\ \tau_2,\ \tau_3$ \\
\hline

\end{tabular}
\label{tabla-resumen}
\end{table}

\begin{table}[h]
\caption{\textbf{Einstein warped $\mathrm{Spin}(7)$-structures}}
\begin{tabular}{|c|c|c|c|c|c|c|c|}\hline
Fiber & $\mathrm{G}_2$-class & $\mu$ & $\mathrm{G_2}$ non-vanishing &  $f(t)$-cone metric & Strict $\mathrm{Spin}(7)$-class & Einstein constant & $\mathrm{Spin}(7)$ non-vanishing \\
&           &           & torsion forms             &       &                  &          $\lambda$        & torsion forms \\ \hline \hline
$\mathcal{NP}$ & $\mathcal{X}_1$ & $6$ & $\tau_0 =\ -4$
& $t$ & $\mathcal{P}$ & $0$ & -- \\
\hline
$\mathcal{NP}$ &$\mathcal{X}_1$ & $6$ & $\tau_0$ & $\sinh t, \ t, \ \sin t$ & $\mathcal{Y}_1$ & $-7,\ 0,\ 7$ & $\lambda_1$ \\
\cline{1-5} \cline{7-7}
$\mathcal{P}$ & $\{0\}$ & $0$ & --  & $e^t$ &  & $-7$ &  \\
\hline
$\nexists$ &  & & &  & $\mathcal{Y}_2$ & & \\
\hline
$\mathcal{LCP}$ & $\mathcal{X}_4$  & $6$ & $\tau_1$ & $\sinh t,\ t,\ \sin t$ & $\mathcal{Y}_1 \oplus \mathcal{Y}_2$ & $-7,\ 0,\ 7$ &
$\lambda_1,\ \lambda_5$ \\
\cline{3-3} \cline{5-5} \cline{7-7}
 &  & $0$ &   & $e^t$ &  & $-7$ &  \\
\cline{3-3} \cline{5-5}
 &  & $-6$ &   & $\cosh t$ &  &  &  \\
\hline
\end{tabular}
\label{tabla-resumen2}
\end{table}

\end{landscape}

\medskip

\end{section}

%%%%%%%%%%%%%%%%%%%%%%%%%%%%%%%%%%%%%%%%%%%%%%%%%%%%%%
%%%%%%%%%%%%%%%%%%%%%%%%%%%%%%%%%%%%%%%%%%%%%%%%%%%%%%

\noindent{\textbf{Acknowledgments}}.
We are grateful to Stefan Ivanov for useful comments about Question~\ref{Q1}.
This work has been partially supported by the project MTM2017-85649-P (AEI/Feder, UE), and
Gobierno de Arag\'on/Fondo Social Europeo--Grupo Consolidado E15 Geometr\'{\i}a.

\medskip

%%%%%%%%%%%%%%%%%%%%%%%%%%%%%%%%%%%%%%%%%%%%%%%%%%%%%%%%%%%%%%%%%%%%%%%%%%%%%%%%%%%%%%
%BIBLIOGRAPHY%%%%%%%%%%%%%%%%%%%%%%%%%%%%%%%%%%%%%%%%%%%%%%%%%%%%%%%%%%%%%%%%%%%%%%%%
%%%%%%%%%%%%%%%%%%%%%%%%%%%%%%%%%%%%%%%%%%%%%%%%%%%%%%%%%%%%%%%%%%%%%%%%%%%%%%%%%%%%%%

\end{document}